\documentclass[preprint,12pt,3p]{elsarticle}
\usepackage{verbatim}
\usepackage{amsmath,amsthm,amsfonts,amssymb}
\usepackage{algorithm}
\usepackage{algpseudocode}
\usepackage{cases}
\usepackage[hidelinks]{hyperref}
\usepackage{graphics}
\usepackage{enumitem}
\usepackage{booktabs}
\usepackage{geometry}
\usepackage{longtable}
\usepackage{tabularx}
\usepackage{graphicx}
\usepackage{booktabs}
\usepackage{float}
\usepackage{caption}
\usepackage{tikz}
\usepackage{subcaption}
\graphicspath{{./figure/}}
\usepackage{multicol}
\usepackage[toc,title,page]{appendix}
\usepackage{fullpage}
\usepackage[normalem]{ulem}
\usepackage{xcolor,sectsty}
\definecolor{astral}{RGB}{46,116,181}
\subsectionfont{\color{astral}}
\sectionfont{\color{astral}}
\usetikzlibrary{calc}
\newtheorem{theorem}{Theorem}[section]

\newtheorem{remark}{Remark}

\newtheorem{definition}[theorem]{Definition}
\newtheorem{proposition}[theorem]{Proposition}
\newtheorem{example}[theorem]{Example}

\journal{be decided}
\definecolor{darkslategray}{rgb}{0.18, 0.31, 0.31}
\definecolor{warmblack}{rgb}{0.0, 0.26, 0.26}

\usetikzlibrary{matrix}
\tikzset{matrixlayer/.style={
matrix of math nodes,nodes in empty cells,
nodes={draw=#1,fill=#1!15,minimum size=10mm,anchor=center,text=black},
column sep=-.5*\pgflinewidth,
row sep=-.5*\pgflinewidth
}}

\tikzstyle{tensor}=[rectangle,draw=blue!50,fill=blue!20,thick]

\begin{document}

\begin{frontmatter}
\title{\textcolor{warmblack}{Iterative Methods for Computing the T-Square Root of Third-Order Tensors}}
\author{ Hemant Sharma, Nachiketa Mishra}

\address{Department of Mathematics,\\
        Indian Institute of Information Technology, Design and Manufacturing Kancheepuram, Chennai, India\\
        {\bf Email:} 
        sharmahemant39@gmail.com;  nmishra@iiitdm.ac.in
        }
                        
\vspace{-2cm}

\begin{abstract}
We develop and analyze iterative methods for computing the principal square root
of third-order tensors under the T-product framework. Tensor extensions of the
Newton iteration (quadratic convergence) and the Denman--Beavers iteration
(geometric convergence with simultaneous computation of the inverse square root)
are proposed, with rigorous convergence guarantees established via the
Fourier-domain block-diagonalization of the T-product. We apply these methods to
image processing, introducing Tensor Decorrelated Grayscale conversion,
T-Whitening, and optimal color transfer under the T-product geometry. We also
formulate the Tensor Bures--Wasserstein distance and prove it defines a valid
metric on the space of T-positive definite tensors. Numerical experiments confirm
rapid convergence and demonstrate that the proposed tensor-based techniques offer
improved structural preservation and cross-channel decorrelation compared to
classical methods.

\medskip
\noindent\textbf{Keywords:}
Tensor Square Root, T-Product, Newton Iteration, Denman--Beavers Iteration,
Bures--Wasserstein Distance, Tensor Whitening, Color Transfer, Image Processing.

\medskip
\noindent\textbf{MSC (2020):}
15A16, 15A69, 65F60, 94A08.
\end{abstract}

\end{frontmatter}

\section{Introduction}\label{sec:intro}

\subsection{Motivation}\label{ssec:motivation}

One of the central challenges in modern data analysis is the extension of
matrix-based tools to higher-order tensor representations, which arise naturally
in applications such as image processing~\cite{kilmer2013third},
computer vision~\cite{hao2013facial}, signal processing~\cite{martin2013order},
and machine learning~\cite{newman2018stable}. In the matrix setting, many
important operations---including spectral decomposition, polar decomposition,
and the computation of matrix functions---have well-developed theory and efficient
algorithms~\cite{higham2008functions}. Extending these operations to tensors in a
mathematically principled way is essential for exploiting the rich structure of
multidimensional data.

A particularly important matrix function is the \emph{square root} of a positive
definite matrix, which plays a central role in several contexts:
\begin{enumerate}
  \item \textbf{The Bures--Wasserstein distance.}\;
        The Bures distance~\cite{bhatia2019bures}, also known as the Wasserstein-2
        distance in the context of Gaussian distributions, provides a geometrically
        meaningful way to compare positive definite matrices. It is formulated
        using the trace and the matrix square root:
        \[
          d_{\mathrm{BW}}(A, B)
          = \Bigl[\operatorname{tr}(A) + \operatorname{tr}(B)
            - 2\,\operatorname{tr}\!\bigl(
            (A^{1/2} B\, A^{1/2})^{1/2}\bigr)\Bigr]^{1/2}.
        \]
        This distance has found wide application in quantum information theory,
        optimal transport, and statistics.

  \item \textbf{Whitening and decorrelation.}\;
        The inverse square root of a covariance matrix is used to whiten data,
        removing second-order correlations and normalizing variance. This is a
        standard preprocessing step in machine learning, signal processing, and
        image analysis.

  \item \textbf{Optimal transport.}\;
        The optimal transport (Monge) map between two zero-mean Gaussian
        distributions with covariances $C_s$ and $C_t$ is given by
        $T = C_s^{-1/2}(C_s^{1/2} C_t\, C_s^{1/2})^{1/2} C_s^{-1/2}$,
        which requires the computation of matrix square roots.
\end{enumerate}

With the increasing prevalence of tensor-valued data in real-world
applications---such as color images, video sequences, hyperspectral imagery, and
spatiotemporal signals---there is a clear need to generalize these operations to
the tensor setting. This motivates the central question of the present work:
\emph{How can one efficiently and reliably compute the square root of a
third-order tensor, and what are the practical consequences for multidimensional
data analysis?}

\subsection{The T-Product Framework}\label{ssec:tproduct}

Among various tensor algebraic frameworks, the T-product, introduced by Kilmer
and Martin~\cite{kilmer2011factorization} and further developed
in~\cite{kilmer2013third, lund2020tensor}, provides an elegant algebraic
structure for operating on third-order tensors. Given two tensors
$\mathcal{A}, \mathcal{B} \in \mathbb{R}^{n \times n \times p}$, their T-product
is defined as
\begin{equation}\label{eq:tprod}
  \mathcal{C} = \mathcal{A} * \mathcal{B}
  = \mathrm{fold}\bigl(\mathrm{bcirc}(\mathcal{A})
    \cdot \mathrm{unfold}(\mathcal{B})\bigr),
\end{equation}
where $\mathrm{bcirc}(\mathcal{A}) \in \mathbb{R}^{np \times np}$ is the
block-circulant matrix formed from the frontal slices of $\mathcal{A}$.

The T-product enables the definition of tensor functions via matrix functions
applied to the block-circulant representation. For a scalar function $f$, the
tensor function $f(\mathcal{A})$ is defined~\cite{lund2020tensor} as
\begin{equation}\label{eq:tfunc}
  f(\mathcal{A})
  = \mathrm{fold}\bigl(f\!\bigl(\mathrm{bcirc}(\mathcal{A})\bigr)
    \cdot \mathrm{unfold}(\mathcal{E}_1)\bigr),
\end{equation}
where $\mathcal{E}_1$ is the tensor whose first frontal slice is the identity
matrix and all other slices are zero.

A key computational advantage of the T-product is that the block-circulant matrix
$\mathrm{bcirc}(\mathcal{A})$ is block-diagonalized by the discrete Fourier
transform (DFT). Specifically, if
$\hat{\mathcal{A}} = \mathrm{fft}(\mathcal{A}, \text{axis}=3)$, then the
T-product reduces to slice-wise matrix multiplication~\cite{lund2020tensor} in the Fourier domain:
\begin{equation}\label{eq:fourier-tprod}
  \hat{C}^{(i)} = \hat{A}^{(i)} \cdot \hat{B}^{(i)},
  \quad i = 1, \ldots, p.
\end{equation}
This decomposition is the linchpin of all algorithms developed in this paper: it
allows us to reduce tensor operations to $p$ independent matrix operations, each
of size $n \times n$.

Traditional matrix-based approaches treat multi-dimensional data as a collection
of independent matrices, often operating slice by slice. While this is
computationally straightforward, it overlooks the inherent multi-way structure
present in tensor data. In contrast, the T-product framework enables structured
operations across all slices simultaneously through the block-circulant
representation and the Fourier domain. This results in a richer algebra that
naturally captures correlations across modes---a property that is particularly
important for tasks like whitening, decorrelation, and distance computation.

\subsection{Problem Statement}\label{ssec:problem}

Given a T-positive definite tensor
$\mathcal{A} \in \mathbb{R}^{n \times n \times p}$, we seek to compute its
\emph{principal T-square root}: the unique T-positive definite tensor
$\mathcal{X} \in \mathbb{R}^{n \times n \times p}$ satisfying
\begin{equation}\label{eq:tsqrt-problem}
  \mathcal{X} * \mathcal{X} = \mathcal{A}.
\end{equation}
By the Fourier-domain reduction~\eqref{eq:fourier-tprod}, this is equivalent to
computing the principal matrix square root of each Fourier-domain frontal slice:
\begin{equation}\label{eq:slice-sqrt}
  \hat{X}^{(i)} = \bigl(\hat{A}^{(i)}\bigr)^{1/2},
  \quad i = 1, \ldots, p,
\end{equation}
followed by the inverse transform
$\mathcal{X} = \mathrm{ifft}(\hat{\mathcal{X}}, \text{axis}=3)$.

While various iterative methods for computing the matrix square root---such as the
Newton method~\cite{higham2008functions},
the Denman--Beavers algorithm~\cite{denman1976matrix}, and
the Schulz iteration~\cite{laasonen1958iterative}---are well-established in the
literature, a direct and structured approach for computing the square root of
third-order tensors has not been formally presented under the T-product framework.
Moreover, the applications of the tensor square root to image processing problems
such as whitening, grayscale conversion, color transfer, and geometric distance
computation have not been systematically explored.

\subsection{Related Work}\label{ssec:related}

The T-product was introduced by Kilmer and
Martin~\cite{kilmer2011factorization} and has since become a widely used framework
for tensor computations. The tensor singular value decomposition (T-SVD) under
this framework was developed in~\cite{kilmer2013third} and has been applied to
facial recognition~\cite{hao2013facial}, image
compression~\cite{martin2013order}, and neural
networks~\cite{newman2018stable}.

The theory of tensor functions under the T-product was formalized by
Lund~\cite{lund2020tensor}, who defined tensor functions via the block-circulant
representation and studied their algebraic properties. The Fr\'{e}chet derivative
of tensor T-functions was subsequently analyzed in~\cite{lund2023frechet}.

In the matrix setting, iterative methods for the square root have a long history.
The Newton iteration for the matrix square root dates back
to~\cite{laasonen1958iterative} and was extensively analyzed by
Higham~\cite{higham2008functions}. The Denman--Beavers
iteration~\cite{denman1976matrix} is a coupled method that simultaneously
computes the square root and its inverse, and is known for its numerical
stability. Comprehensive treatments of matrix square root algorithms can be found
in~\cite{higham2008functions}.

The Bures--Wasserstein distance between positive definite matrices was studied
in~\cite{bhatia2019bures}, where its metric properties and connections to optimal
transport were established. However, a tensor generalization of this distance
under the T-product framework, along with a systematic treatment of the tensor
square root and its applications to image processing, has been lacking.

\subsection{Contributions}\label{ssec:contributions}

The main contributions of this work are as follows:

\begin{enumerate}
  \item \textbf{Tensor Newton iteration (Section~\ref{sec:methods}).}\;
        We formulate the Newton iteration for computing the principal T-square
        root of a third-order tensor under the T-product and prove that it
        converges quadratically (Theorem~\ref{thm:newton-strengthened}).

 \item \textbf{Tensor Denman--Beavers iteration (Section~\ref{sec:methods}).}
 We extend the Denman--Beavers coupled iteration to the tensor setting and
 establish geometric convergence (Theorem~\ref{thm:db-conv}). This method
 simultaneously produces both $\mathcal{A}^{1/2}$ and $\mathcal{A}^{-1/2}$,
 which is advantageous for applications requiring both quantities.
  \item \textbf{Tensor Bures--Wasserstein distance
        (Section~\ref{sec:applications}).}\;
        We define the Tensor Bures--Wasserstein (TBW) distance on the space of
        T-positive definite tensors and prove that it is a valid metric
        (Proposition~\ref{prop:tbw-metric}). This provides a principled geometric
        measure for comparing tensor-valued data such as covariance tensors of
        multi-channel images.

  \item \textbf{Tensor Decorrelated Grayscale conversion
        (Section~\ref{sec:applications}).}\;
        We propose a grayscale conversion method that adapts to the statistical
        structure of the image through the inverse tensor square root of the
        covariance tensor, yielding improved edge preservation and contrast
        compared to fixed-weight methods.

  \item \textbf{T-Whitening (Section~\ref{sec:applications}).}\;
        We develop a tensor whitening procedure that preserves the
        multidimensional structure of image data, achieving decorrelation quality
        comparable to matrix whitening while avoiding the information loss
        inherent in flattening-based approaches.

  \item \textbf{Tensor color transfer (Section~\ref{sec:applications}).}\;
        We formulate an optimal color transfer method that computes the Monge
        transport map between source and target color distributions under the
        T-product geometry, naturally accounting for cross-channel correlations.

  \item \textbf{Comprehensive numerical experiments
        (Section~\ref{sec:applications}).}\;
        We provide detailed numerical comparisons of the proposed methods,
        including convergence analysis, wall-clock timings, and quantitative
        evaluation of image processing quality using SSIM(Structural Similarity Index Measure), edge preservation
        metrics, and decorrelation indices.
\end{enumerate}

\subsection{Organization of the Paper}\label{ssec:organization}

The remainder of this paper is organized as follows.
Section~\ref{sec:methods} develops the iterative algorithms for computing the
tensor T-square root. We begin with the necessary preliminaries on the T-SVD and
tensor square root (Section~\ref{ssec:tsvd}), then present the Newton iteration
(Section~\ref{ssec:newton}) and the Denman--Beavers iteration
(Section~\ref{ssec:db}), each accompanied by convergence proofs and numerical
examples. A comparison of the two methods is provided in
Section~\ref{ssec:method-comparison}.

Section~\ref{sec:applications} demonstrates the practical applications of the
tensor square root in image processing. We define the tensor covariance
(Section~\ref{ssec:tensor-cov}), formulate the Tensor Bures--Wasserstein distance
(Section~\ref{ssec:tbw}), and develop algorithms for Tensor Decorrelated Grayscale
conversion (Section~\ref{ssec:tgrayscale}), T-Whitening
(Section~\ref{ssec:twhite}), and tensor color transfer
(Section~\ref{ssec:color-transfer}). Comprehensive numerical comparisons are
presented in Sections~\ref{ssec:comparison}--\ref{ssec:classical-compare}.

\noindent Section~\ref{sec:conclusion} summarizes the key findings and discusses directions
for future research.

\subsection{Notation}\label{ssec:notation}

We collect the notation used throughout the paper for reference. Scalars are
denoted by lowercase letters ($a, b, \ldots$), vectors by bold lowercase
($\mathbf{a}, \mathbf{b}, \ldots$), matrices by uppercase letters ($A, B, \ldots$),
and third-order tensors by calligraphic letters
($\mathcal{A}, \mathcal{B}, \ldots$). The frontal slices of a tensor
$\mathcal{A} \in \mathbb{R}^{n \times m \times p}$ are denoted
$\mathcal{A}(:,:,k)$ or $A^{(k)}$ for $k = 1, \ldots, p$. The Fourier transform
of $\mathcal{A}$ along the third mode is denoted
$\hat{\mathcal{A}} = \mathrm{fft}(\mathcal{A}, \text{axis}=3)$, with frontal
slices $\hat{A}^{(i)}$. The T-product is denoted by $*$, the T-transpose by
$\mathcal{A}^{\top}$, and the T-product inverse by $\mathcal{A}^{-1}$. The identity tensor under the T-product
is denoted $\mathcal{I}$. The tensor Frobenius norm is defined as
$\|\mathcal{A}\|_F = \bigl(\sum_{i,j,k} a_{ijk}^2\bigr)^{1/2}$.
The tensor trace is denoted $\operatorname{Ttr}(\cdot)$.

\section{Iterative Methods for the Tensor T-Square Root}\label{sec:methods}

In this section, we develop iterative algorithms for computing the principal
square root of a third-order tensor under the T-product framework. We begin by
recalling the relevant tensor decompositions and then present two methods---the
Newton iteration and the Denman--Beavers iteration---together with rigorous
convergence analyses and illustrative numerical examples.

\subsection{Preliminaries: T-SVD and Tensor Square Root}\label{ssec:tsvd}

Let $\mathcal{A} \in \mathbb{R}^{n \times n \times p}$ be a third-order tensor.
Its tensor singular value decomposition (T-SVD) under the T-product framework is
\begin{equation}\label{eq:tsvd}
  \mathcal{A} = \mathcal{U} * \mathcal{S} * \mathcal{V}^{*},
\end{equation}
where $\mathcal{U}, \mathcal{V} \in \mathbb{R}^{n \times n \times p}$ are
orthogonal tensors and $\mathcal{S} \in \mathbb{R}^{n \times n \times p}$ is an
\emph{f-diagonal} tensor whose frontal slices contain the singular values of
$\mathcal{A}$, computed in the Fourier domain. This decomposition generalizes the
classical matrix SVD to the tensor setting using the algebra of the
T-product~\cite{kilmer2011factorization}.

The square root of the f-diagonal tensor $\mathcal{S}$, denoted
$\mathcal{S}^{1/2}$, is constructed slice-wise in the Fourier domain
following the T-SVD framework and tensor t-function
formalism~\cite{kilmer2011factorization, kilmer2013third, lund2020tensor}.
Let $\sigma_j^{(i)} > 0$ for $j = 1, \ldots, n$ and $i = 1, \ldots, p$
denote the singular values appearing on the diagonal of the $i$-th
Fourier-domain frontal slice $\hat{S}^{(i)}$. Then $\mathcal{S}^{1/2}$
is defined via
\begin{equation}\label{eq:sqrtS}
	\bigl(\widehat{\mathcal{S}^{1/2}}\,\bigr)^{(i)}
	= \operatorname{diag}\!\left(
	\sqrt{\sigma_1^{(i)}},\;
	\sqrt{\sigma_2^{(i)}},\; \ldots,\;
	\sqrt{\sigma_n^{(i)}}
	\right),
	\qquad i = 1, \ldots, p,
\end{equation}
followed by the inverse FFT along the third mode,
$\mathcal{S}^{1/2} = \mathrm{ifft}\bigl(\widehat{\mathcal{S}^{1/2}},\,
\text{axis}=3\bigr)$. Using this, a T-positive semi-definite tensor
$\mathcal{A}$ can be factored as
$\mathcal{A} = \mathcal{M} * \mathcal{M}^{\top}$, where
$\mathcal{M} = \mathcal{U} * \mathcal{S}^{1/2}$.

\begin{definition}[T-Positive Definite Tensor]\label{def:tpd}
A tensor $\mathcal{A} \in \mathbb{R}^{n \times n \times p}$ is said to be
\emph{T-positive definite} if every frontal slice of its Fourier transform,
\[
  \hat{A}^{(i)} = \bigl(\mathrm{fft}(\mathcal{A}, \text{axis}=3)\bigr)^{(i)},
  \quad i = 1, \ldots, p,
\]
is a symmetric positive definite matrix.
\end{definition}

\begin{definition}[Principal T-Square Root]\label{def:tsqrt}
Let $\mathcal{A} \in \mathbb{R}^{n \times n \times p}$ be a T-positive definite
tensor. The \emph{principal T-square root} of $\mathcal{A}$ is the unique
T-positive definite tensor $\mathcal{A}^{1/2} \in \mathbb{R}^{n \times n \times p}$
satisfying
\begin{equation}\label{eq:tsqrt-def}
  \mathcal{A}^{1/2} * \mathcal{A}^{1/2} = \mathcal{A}.
\end{equation}
\end{definition}

\begin{remark}[Fourier-domain reduction]\label{rem:fourier}
Since the T-product is diagonalized by the discrete Fourier transform (DFT) along
the third mode, computing the T-square root reduces to computing matrix square
roots of the frontal slices in the Fourier domain:
\[
  \hat{X}^{(i)} = \bigl(\hat{A}^{(i)}\bigr)^{1/2},
  \quad i = 1, \ldots, p,
\]
followed by the inverse transform
$\mathcal{X} = \mathrm{ifft}(\hat{\mathcal{X}}, \text{axis}=3)$.
All iterative methods developed in this section exploit this reduction.
\end{remark}

\subsection{Newton Iteration for the T-Square Root}\label{ssec:newton}

The Newton iteration for the matrix square root~\cite{higham2008functions} is a
classical method with quadratic convergence. We now extend it to the tensor
setting under the T-product.

Given a T-positive definite tensor $\mathcal{A} \in \mathbb{R}^{n \times n \times p}$,
the Newton iteration for computing $\mathcal{A}^{1/2}$ is defined by
\begin{equation}\label{eq:newton-iter}
	\mathcal{X}_{k+1}
	= \frac{1}{2}\bigl(\mathcal{X}_k + \mathcal{X}_k^{-1} * \mathcal{A}\bigr),
	\qquad k = 0, 1, 2, \ldots,
\end{equation}
with initial tensor $\mathcal{X}_0 = \mathcal{I}$, where $\mathcal{X}_k^{-1}$
denotes the T-product inverse of $\mathcal{X}_k$ and $\mathcal{I} \in
\mathbb{R}^{n \times n \times p}$ is the identity tensor (whose first frontal
slice is $I_n$ and remaining slices are zero under the T-product convention).

As described in Remark~\ref{rem:fourier}, the iteration~\eqref{eq:newton-iter} is
implemented slice-wise in the Fourier domain. After transforming $\mathcal{A}$ via
the FFT, each frontal slice is updated independently:
\[
  \hat{X}^{(k)}_{j+1}
  = \frac{1}{2}\left(\hat{X}^{(k)}_j
    + \hat{A}^{(k)}\bigl(\hat{X}^{(k)}_j\bigr)^{-1}\right),
  \quad k = 1, \ldots, p,
\]
where subscript $j$ denotes the iteration index. The result is then recovered via
the inverse FFT.

\begin{algorithm}[H]
\caption{Newton Iteration for Tensor T-Square Root}
\label{alg:newton}
\begin{algorithmic}[1]
\Require Tensor $\mathcal{A} \in \mathbb{R}^{n \times n \times p}$, maximum
         iterations $N_{\max}$, tolerance $\varepsilon > 0$
\Ensure  Tensor square root $\mathcal{X}$ such that
         $\mathcal{X} * \mathcal{X} \approx \mathcal{A}$
\State $\hat{\mathcal{A}} \leftarrow \mathrm{fft}(\mathcal{A}, \text{axis}=3)$
\For{$k = 1$ to $p$}
  \State $\hat{X}^{(k)} \leftarrow \hat{A}^{(k)}$
  \Comment{Initialize each slice with $\hat{A}^{(k)}$}
\EndFor
\For{$j = 1$ to $N_{\max}$}
  \For{$k = 1$ to $p$}
    \State $\hat{X}^{(k)} \leftarrow \dfrac{1}{2}\!\left(\hat{X}^{(k)}
           + \hat{A}^{(k)}\bigl(\hat{X}^{(k)}\bigr)^{-1}\right)$
  \EndFor
  \State Compute residual:
         $r_j = \left(\sum_{k=1}^{p}
         \bigl\|\hat{X}^{(k)}\hat{X}^{(k)} - \hat{A}^{(k)}\bigr\|_F^2
         \right)^{1/2}$
  \If{$r_j < \varepsilon$}
    \State \textbf{break}
  \EndIf
\EndFor
\State $\mathcal{X} \leftarrow \mathrm{ifft}(\hat{\mathcal{X}}, \text{axis}=3)$
\State \Return $\mathcal{X}$
\end{algorithmic}
\end{algorithm}

\subsubsection{Convergence Analysis of the Newton Iteration}
\label{sssec:newton-convergence}

\begin{theorem}[Convergence of the Newton Iteration]\label{thm:newton-strengthened}
	Let $\mathcal{A} \in \mathbb{R}^{n \times n \times p}$ be T-positive definite with Fourier slices $\hat{A}^{(i)}$, $i = 1, \ldots, p$. Denote the eigenvalues of $\hat{A}^{(i)}$ by $\lambda_j^{(i)} > 0$, and define
	\[
	\lambda_{\min} := \min_{i,j} \lambda_j^{(i)}, \qquad \lambda_{\max} := \max_{i,j} \lambda_j^{(i)}.
	\]
	Consider the Newton iteration~(9) implemented in the Fourier domain with initialization $\hat{X}_0^{(i)} = \hat{A}^{(i)}$ for each $i = 1, \ldots, p$ (as in Algorithm~1). Then:
	\begin{enumerate}[label=(\alph*)]
		\item \textbf{(Well-definedness)} All iterates $\mathcal{X}_k$ are T-positive definite and hence T-invertible.
		\item \textbf{(Global convergence)} The sequence $\{\mathcal{X}_k\}$ converges to $\mathcal{A}^{1/2}$ in the tensor Frobenius norm.
		\item \textbf{(Quadratic convergence)} There exists a constant $C > 0$ depending on $\lambda_{\min}$ and $\lambda_{\max}$ such that
		\[
		\|\mathcal{X}_{k+1} - \mathcal{A}^{1/2}\|_F \leq C \|\mathcal{X}_k - \mathcal{A}^{1/2}\|_F^2
		\]
		for all $k \geq 0$.
	\end{enumerate}
\end{theorem}

\begin{proof}
	The proof proceeds by reducing to the Fourier domain and applying strengthened versions of the classical matrix Newton iteration analysis.
	
	\medskip
	\noindent\textbf{Step 1: Fourier-domain decoupling.}
	Applying the DFT along the third mode, the tensor Newton iteration decouples into $p$ independent matrix Newton iterations: for each $i = 1, \ldots, p$,
	\[
	\hat{X}_{k+1}^{(i)} = \frac{1}{2}\left(\hat{X}_k^{(i)} + \hat{A}^{(i)} \left(\hat{X}_k^{(i)}\right)^{-1}\right), \qquad \hat{X}_0^{(i)} = \hat{A}^{(i)}.
	\]
	
	\noindent\textbf{Step 2: Slice-wise analysis with initialization $X_0 = A$.}
	For each slice $i$, $\hat{A}^{(i)}$ is symmetric positive definite with eigenvalues $\lambda_1^{(i)} \geq \cdots \geq \lambda_n^{(i)} > 0$. The matrix Newton iteration for the square root with initialization $X_0 = A$ (rather than $X_0 = I$) is a well-studied variant. By Higham~\cite{higham2008functions}, Section~6.3, Theorem~6.9, the iteration
	\[
	X_{k+1} = \frac{1}{2}(X_k + A X_k^{-1}), \qquad X_0 = A,
	\]
	applied to a symmetric positive definite matrix $A$ satisfies:
	\begin{itemize}
		\item All iterates $X_k$ are symmetric positive definite.
		\item $X_k \to A^{1/2}$ as $k \to \infty$.
		\item The convergence is quadratic: there exists $C_i > 0$ such that
		\[
		\|\hat{X}_k^{(i)} - (\hat{A}^{(i)})^{1/2}\|_F \leq C_i \|\hat{X}_{k-1}^{(i)} - (\hat{A}^{(i)})^{1/2}\|_F^2 \quad \text{for all } k \geq 0.
		\]
	\end{itemize}
	
	To verify that $X_0 = A$ lies within the basin of quadratic convergence, we use the scaled Newton iteration analysis. Define $Z_k^{(i)} = (\hat{A}^{(i)})^{-1/2} \hat{X}_k^{(i)} (\hat{A}^{(i)})^{-1/2}$. Then $Z_0^{(i)} = (\hat{A}^{(i)})^{1/2}$ is SPD, and the iteration on $Z_k^{(i)}$ becomes the standard matrix sign iteration, which converges quadratically for any SPD initial matrix (see~\cite{higham2008functions}, Theorem~5.2).
	
	\medskip
	\noindent\textbf{Step 3: Global tensor bound.}
	The slice-wise convergence constants satisfy $C_i = \frac{1}{2} \|(\hat{A}^{(i)})^{1/2}\| \cdot \|(\hat{A}^{(i)})^{-1/2}\|^2$, which depends on the condition number $\kappa_i = \lambda_{\max}^{(i)} / \lambda_{\min}^{(i)}$. Define
	\[
	C := \max_{1 \leq i \leq p} C_i = \frac{1}{2} \cdot \frac{\lambda_{\max}^{1/2}}{\lambda_{\min}}.
	\]
	Using the Parseval-type identity $\|\mathcal{T}\|_F^2 = \sum_{i=1}^p \|\hat{T}^{(i)}\|_F^2$, we obtain
	\begin{align*}
		\|\mathcal{X}_{k+1} - \mathcal{A}^{1/2}\|_F^2 
		&= \sum_{i=1}^p \|\hat{X}_{k+1}^{(i)} - (\hat{A}^{(i)})^{1/2}\|_F^2 \\
		&\leq \sum_{i=1}^p C_i^2 \|\hat{X}_k^{(i)} - (\hat{A}^{(i)})^{1/2}\|_F^4 \\
		&\leq C^2 \sum_{i=1}^p \|\hat{X}_k^{(i)} - (\hat{A}^{(i)})^{1/2}\|_F^4 \\
		&\leq C^2 \left(\sum_{i=1}^p \|\hat{X}_k^{(i)} - (\hat{A}^{(i)})^{1/2}\|_F^2\right)^2 \\
		&= C^2 \|\mathcal{X}_k - \mathcal{A}^{1/2}\|_F^4,
	\end{align*}
	where the last inequality uses $\sum a_i^2 \leq (\sum a_i)^2$ for non-negative $a_i$ (since $\|x\|_4 \leq \|x\|_2$ in $\mathbb{R}^p$). Taking square roots yields
	\[
	\|\mathcal{X}_{k+1} - \mathcal{A}^{1/2}\|_F \leq C \|\mathcal{X}_k - \mathcal{A}^{1/2}\|_F^2,
	\]
	establishing quadratic convergence from the first iteration.
\end{proof}

\smallskip
\noindent\textit{Remark on initialization.}
The above analysis assumes $\mathcal{X}_0 = \mathcal{I}$ for theoretical clarity.
For the practical initialization $\mathcal{X}_0 = \mathcal{A}$ used in
Algorithm~\ref{alg:newton}, the same quadratic convergence holds since
$\mathcal{A}$ lies closer to $\mathcal{A}^{1/2}$ than $\mathcal{I}$ for
T-positive definite tensors with eigenvalues bounded away from zero, ensuring
that the iterates enter the basin of quadratic convergence immediately.

\begin{remark}[Choice of initialization]\label{rem:init}
Algorithm~\ref{alg:newton} initializes each Fourier slice with
$\hat{X}^{(k)}_0 = \hat{A}^{(k)}$ rather than the identity matrix. This is
because, for positive definite matrices, the initialization $X_0 = A$ provides a
better starting point that already lies in the basin of quadratic convergence for
a wider class of inputs. The identity initialization $X_0 = I$ is used in the
theoretical analysis for clarity, but in practice $X_0 = A$ is preferred and is
used in all our numerical experiments.
\end{remark}

\subsubsection{Numerical Example}\label{sssec:newton-example}

\begin{example}\label{ex:newton}
Let $\mathcal{A} \in \mathbb{R}^{3 \times 3 \times 3}$ be the T-positive
definite tensor defined by
\[
  \mathcal{A}(:,:,1) = \begin{bmatrix}
    3 & 1 & 0 \\ 1 & 4 & 1 \\ 0 & 1 & 3
  \end{bmatrix}, \quad
  \mathcal{A}(:,:,2) = \begin{bmatrix}
    2 & 0.5 & 0 \\ 0.5 & 2 & 0.5 \\ 0 & 0.5 & 2
  \end{bmatrix}, \quad
  \mathcal{A}(:,:,3) = \begin{bmatrix}
    1 & 0 & 0 \\ 0 & 1 & 0 \\ 0 & 0 & 1
  \end{bmatrix}.
\]
Applying Algorithm~\ref{alg:newton} with $N_{\max} = 10$ and tolerance
$\varepsilon = 10^{-6}$, convergence is achieved in 5 iterations.
The computed T-square root $\mathcal{X} = \mathcal{A}^{1/2}$ has frontal slices:
\[
  \mathcal{X}(:,:,1) = \begin{bmatrix}
    1.63499 & 0.29426 & -0.02899 \\
    0.29426 & 1.90498 &  0.29426 \\
   -0.02899 & 0.29426 &  1.63499
  \end{bmatrix},
\]
\[
  \mathcal{X}(:,:,2) = \begin{bmatrix}
    0.58655 & 0.05896 & -0.00286 \\
    0.05896 & 0.49317 &  0.05896 \\
   -0.00286 & 0.05896 &  0.58655
  \end{bmatrix},
\]
\[
  \mathcal{X}(:,:,3) = \begin{bmatrix}
    0.20962 & -0.05470 &  0.01352 \\
   -0.05470 &  0.21371 & -0.05470 \\
    0.01352 & -0.05470 &  0.20962
  \end{bmatrix}.
\]
The residual satisfies
$\|\mathcal{X} * \mathcal{X} - \mathcal{A}\|_F = 6.89 \times 10^{-15}$, which is
near machine precision.
\end{example}

\smallskip
\noindent\textbf{Diagnosing the convergence rate.}
Throughout our numerical experiments, we track two ratios of successive
residuals to diagnose the empirical rate of convergence, following the
standard Q-order framework~\cite{nocedal2006numerical}. Given the residual
sequence $r_k = \|\mathcal{X}_k * \mathcal{X}_k - \mathcal{A}\|_F$, we report:
\begin{itemize}
	\item The \emph{first-order ratio} $\rho_k := r_k / r_{k-1}$. If the
	iteration is Q-linearly convergent with rate $\rho \in (0,1)$, then
	$\rho_k \to \rho$ as $k \to \infty$. A ratio that decreases steadily
	toward zero indicates \emph{Q-superlinear} convergence.
	\item The \emph{second-order ratio} $q_k := r_k / r_{k-1}^{2}$. If the
	iteration is Q-quadratically convergent, i.e.\ $r_k \leq C\,r_{k-1}^{2}$
	for some constant $C > 0$, then $q_k$ stays bounded by $C$. A bounded
	$q_k$ together with a rapidly decreasing $\rho_k$ is the empirical
	signature of Q-quadratic convergence.
\end{itemize}
Conversely, if $\rho_k$ settles at a roughly constant value (and $q_k$ grows),
the observed behavior is only Q-linear.

The convergence history of the Newton iteration for Example~\ref{ex:newton} is
shown in Table~\ref{tab:newton-conv}. The
residual decreases by roughly two orders of magnitude per iteration after the
initial phase, confirming the quadratic convergence established in
Theorem~\ref{thm:newton-strengthened}.
\begin{table}[H]
	\centering
\caption{Convergence of the Newton iteration (Algorithm~\ref{alg:newton})
	for Example~\ref{ex:newton}. The residual is
	$r_k = \|\mathcal{X}_k * \mathcal{X}_k - \mathcal{A}\|_F$; see the
	``Diagnosing the convergence rate'' paragraph preceding this table for the
	interpretation of the ratios $\rho_k$ and $q_k$. The bounded $q_k$ together
	with rapidly decreasing $\rho_k$ confirms Q-quadratic convergence, as
	established in Theorem~\ref{thm:newton-strengthened}.}
	\label{tab:newton-conv}
	\begin{tabular}{cccc}
		\toprule
		Iteration $k$ & Residual $r_k$ & Ratio $r_k/r_{k-1}$ & Ratio $r_k / r_{k-1}^2$ \\
		\midrule
		0 & $6.95 \times 10^{0}$  & ---                    & --- \\
		1 & $5.28 \times 10^{-1}$ & $7.60 \times 10^{-2}$  & $1.09 \times 10^{-2}$ \\
		2 & $5.52 \times 10^{-2}$ & $1.05 \times 10^{-1}$  & $1.98 \times 10^{-1}$ \\
		3 & $7.80 \times 10^{-4}$ & $1.41 \times 10^{-2}$  & $2.56 \times 10^{-1}$ \\
		4 & $1.61 \times 10^{-7}$ & $2.06 \times 10^{-4}$  & $2.65 \times 10^{-1}$ \\
		5 & $6.89 \times 10^{-15}$& $4.28 \times 10^{-8}$  & $2.66 \times 10^{-1}$ \\
		\bottomrule
	\end{tabular}
\end{table}

\subsection{Denman--Beavers Iteration for the T-Square Root}\label{ssec:db}

The Denman--Beavers (DB) iteration~\cite{denman1976matrix} is a coupled iteration
that simultaneously computes both the square root and its inverse. We extend it to
the T-product setting as follows.

Given a T-positive definite tensor
$\mathcal{A} \in \mathbb{R}^{n \times n \times p}$, the DB iteration is defined by
the coupled recursion:
\begin{equation}\label{eq:db-iter}
	\mathcal{X}_{k+1}
	= \frac{1}{2}\bigl(\mathcal{X}_k + \mathcal{Y}_k^{-1}\bigr), \qquad
	\mathcal{Y}_{k+1}
	= \frac{1}{2}\bigl(\mathcal{Y}_k + \mathcal{X}_k^{-1}\bigr),
\end{equation}
with initial values $\mathcal{X}_0 = \mathcal{A}$ and
$\mathcal{Y}_0 = \mathcal{I}$, where $\mathcal{Y}_k^{-1}$ and
$\mathcal{X}_k^{-1}$ denote the T-product inverses.

A key advantage of the DB iteration over the Newton iteration is that it
simultaneously produces approximations to both $\mathcal{A}^{1/2}$ (via the
$\mathcal{X}_k$ sequence) and $\mathcal{A}^{-1/2}$ (via the $\mathcal{Y}_k$
sequence). This is particularly useful in applications such as whitening and color
transfer (Section~\ref{sec:applications}), where both quantities are required.

As with the Newton iteration, the DB iteration is implemented slice-wise in the
Fourier domain. For each frontal slice $k = 1, \ldots, p$:
\begin{equation}\label{eq:db-slice}
  \hat{X}^{(k)}_{j+1}
  = \frac{1}{2}\bigl(\hat{X}^{(k)}_j
    + (\hat{Y}^{(k)}_j)^{-1}\bigr), \qquad
  \hat{Y}^{(k)}_{j+1}
  = \frac{1}{2}\bigl(\hat{Y}^{(k)}_j
    + (\hat{X}^{(k)}_j)^{-1}\bigr).
\end{equation}

\begin{algorithm}[H]
\caption{Denman--Beavers Iteration for Tensor T-Square Root}
\label{alg:db}
\begin{algorithmic}[1]
\Require Tensor $\mathcal{A} \in \mathbb{R}^{n \times n \times p}$, maximum
         iterations $N_{\max}$, tolerance $\varepsilon > 0$
\Ensure  Tensor square root $\mathcal{X} \approx \mathcal{A}^{1/2}$ and inverse
         square root $\mathcal{Y} \approx \mathcal{A}^{-1/2}$
\State $\hat{\mathcal{A}} \leftarrow \mathrm{fft}(\mathcal{A}, \text{axis}=3)$
\For{$k = 1$ to $p$}
  \State $\hat{X}^{(k)} \leftarrow \hat{A}^{(k)}$
  \Comment{Initialize $\mathcal{X}_0 = \mathcal{A}$}
  \State $\hat{Y}^{(k)} \leftarrow I$
  \Comment{Initialize $\mathcal{Y}_0 = \mathcal{I}$}
\EndFor
\For{$j = 1$ to $N_{\max}$}
  \For{$k = 1$ to $p$}
    \State $\hat{X}^{(k)}_{\mathrm{new}} \leftarrow \dfrac{1}{2}\!\left(
           \hat{X}^{(k)} + \bigl(\hat{Y}^{(k)}\bigr)^{-1}\right)$
    \State $\hat{Y}^{(k)}_{\mathrm{new}} \leftarrow \dfrac{1}{2}\!\left(
           \hat{Y}^{(k)} + \bigl(\hat{X}^{(k)}\bigr)^{-1}\right)$
    \State $\hat{X}^{(k)} \leftarrow \hat{X}^{(k)}_{\mathrm{new}}$, \;
           $\hat{Y}^{(k)} \leftarrow \hat{Y}^{(k)}_{\mathrm{new}}$
  \EndFor
  \State Compute residual:
         $r_j = \left(\sum_{k=1}^{p}
         \bigl\|\hat{X}^{(k)}\hat{X}^{(k)} - \hat{A}^{(k)}\bigr\|_F^2
         \right)^{1/2}$
  \If{$r_j < \varepsilon$}
    \State \textbf{break}
  \EndIf
\EndFor
\State $\mathcal{X} \leftarrow \mathrm{ifft}(\hat{\mathcal{X}}, \text{axis}=3)$
\State $\mathcal{Y} \leftarrow \mathrm{ifft}(\hat{\mathcal{Y}}, \text{axis}=3)$
\State \Return $\mathcal{X}$, $\mathcal{Y}$
\end{algorithmic}
\end{algorithm}

\subsubsection{Convergence Analysis of the Denman--Beavers Iteration}
\label{sssec:db-convergence}

\begin{theorem}\label{thm:db-conv}
Let $\mathcal{A} \in \mathbb{R}^{n \times n \times p}$ be T-positive definite.
Consider the Denman--Beavers iteration~\eqref{eq:db-iter} with initial values
$\mathcal{X}_0 = \mathcal{A}$ and $\mathcal{Y}_0 = \mathcal{I}$. Then:
\begin{enumerate}
  \item The sequences $\{\mathcal{X}_k\}$ and $\{\mathcal{Y}_k\}$ are well-defined
        (all iterates remain T-invertible).
  \item $\mathcal{X}_k \to \mathcal{A}^{1/2}$ and
        $\mathcal{Y}_k \to \mathcal{A}^{-1/2}$ as $k \to \infty$.
  \item The convergence is at least geometric (linear) in the tensor Frobenius
        norm: there exist constants $C > 0$ and $r \in (0,1)$ such that for all
        sufficiently large $k$,
        \begin{equation}\label{eq:db-linear}
          \bigl\|\mathcal{X}_k - \mathcal{A}^{1/2}\bigr\|_F \leq C\,r^k,
          \qquad
          \bigl\|\mathcal{Y}_k - \mathcal{A}^{-1/2}\bigr\|_F \leq C\,r^k.
        \end{equation}
\end{enumerate}
\end{theorem}

\begin{proof}
The proof proceeds by reducing the tensor iteration to $p$ independent matrix
iterations via the Fourier transform, and then applying classical results.

\medskip
\noindent\textbf{Step 1: Fourier-domain decoupling.}\;
Apply the DFT along the third mode. Denote
\[
  \hat{\mathcal{X}}_k = \mathrm{fft}(\mathcal{X}_k, \text{axis}=3), \qquad
  \hat{\mathcal{Y}}_k = \mathrm{fft}(\mathcal{Y}_k, \text{axis}=3),
\]
and let $\hat{X}^{(i)}_k$, $\hat{Y}^{(i)}_k$, $\hat{A}^{(i)}$ denote their
$i$-th frontal slices. Since the T-product maps to slice-wise matrix
multiplication under the FFT, the tensor iteration~\eqref{eq:db-iter} decouples
into $p$ independent matrix Denman--Beavers iterations:
\begin{equation}\label{eq:db-slice-iter}
  \hat{X}^{(i)}_{k+1}
  = \frac{1}{2}\bigl(\hat{X}^{(i)}_k + (\hat{Y}^{(i)}_k)^{-1}\bigr),
  \qquad
  \hat{Y}^{(i)}_{k+1}
  = \frac{1}{2}\bigl(\hat{Y}^{(i)}_k + (\hat{X}^{(i)}_k)^{-1}\bigr),
\end{equation}
with $\hat{X}^{(i)}_0 = \hat{A}^{(i)}$ and $\hat{Y}^{(i)}_0 = I$ for each
$i = 1, \ldots, p$.

\medskip
\noindent\textbf{Step 2: Slice-wise convergence.}\;
By the T-positive definiteness of $\mathcal{A}$
(Definition~\ref{def:tpd}), each $\hat{A}^{(i)}$ is a symmetric positive definite
matrix. The classical theory of the matrix Denman--Beavers iteration (see
Higham~\cite{higham2008functions}, Chapter~6) guarantees that for each slice $i$:
\begin{enumerate}
  \item[(a)] All iterates $\hat{X}^{(i)}_k$ and $\hat{Y}^{(i)}_k$ are
             well-defined and invertible.
  \item[(b)] $\hat{X}^{(i)}_k \to (\hat{A}^{(i)})^{1/2}$ and
             $\hat{Y}^{(i)}_k \to (\hat{A}^{(i)})^{-1/2}$ as $k \to \infty$.
  \item[(c)] There exist slice-dependent constants $C_i > 0$,
             $r_i \in (0,1)$, and an index $k_{0,i}$ such that for all
             $k \geq k_{0,i}$:
             \[
               \bigl\|\hat{X}^{(i)}_{k} - (\hat{A}^{(i)})^{1/2}\bigr\|_F
               \leq C_i\,r_i^k, \qquad
               \bigl\|\hat{Y}^{(i)}_{k} - (\hat{A}^{(i)})^{-1/2}\bigr\|_F
               \leq C_i\,r_i^k.
             \]
\end{enumerate}

\medskip
\noindent\textbf{Step 3: Global tensor bound.}\;
The tensor Frobenius norm is related to the Fourier-domain slices by
\[
  \|\mathcal{T}\|_F
  = \left(\sum_{i=1}^{p}
    \bigl\|\hat{T}^{(i)}\bigr\|_F^2\right)^{1/2}.
\]
Define the global constants
\[
  r := \max_{1 \leq i \leq p} r_i \in (0,1), \qquad
  C := \left(\sum_{i=1}^{p} C_i^2\right)^{1/2},
  \qquad
  k_0 := \max_{1 \leq i \leq p} k_{0,i}.
\]
Then for all $k \geq k_0$:
\begin{align}
  \bigl\|\mathcal{X}_k - \mathcal{A}^{1/2}\bigr\|_F^2
  &= \sum_{i=1}^{p}
     \bigl\|\hat{X}^{(i)}_k - (\hat{A}^{(i)})^{1/2}\bigr\|_F^2
  \leq \sum_{i=1}^{p} C_i^2\,r_i^{2k}
  \leq C^2\,r^{2k}, \notag
\end{align}
and therefore
$\|\mathcal{X}_k - \mathcal{A}^{1/2}\|_F \leq C\,r^k$.
An identical argument applies to the $\mathcal{Y}_k$ sequence. This proves the
geometric convergence of both sequences.
\end{proof}

\begin{remark}[On the convergence rate]\label{rem:db-rate}
The Denman--Beavers iteration is known to be locally quadratically
convergent for positive definite matrices (see~\cite{higham2008functions},
Theorem~6.9). In Theorem~\ref{thm:db-conv}, we have stated only the weaker
geometric (linear) convergence guarantee, as this is the global result that holds
without additional assumptions on the spectral spread of $\mathcal{A}$.
In practice, quadratic convergence is typically observed once the iterates
enter a sufficiently small neighborhood of the solution, as illustrated in our
numerical experiments.
\end{remark}

\subsubsection{Numerical Example}\label{sssec:db-example}

\begin{example}\label{ex:db}
We apply Algorithm~\ref{alg:db} to the same tensor
$\mathcal{A} \in \mathbb{R}^{3 \times 3 \times 3}$ from Example~\ref{ex:newton},
with $N_{\max} = 10$ and tolerance $\varepsilon = 10^{-12}$. Convergence is
reached at iteration 6 with residual
$\|\mathcal{X}_6 * \mathcal{X}_6 - \mathcal{A}\|_F = 2.46 \times 10^{-15}$.

The computed T-square root agrees with the Newton result to machine precision:
\[
  \mathcal{X}(:,:,1) = \begin{bmatrix}
    1.63499 & 0.29426 & -0.02899 \\
    0.29426 & 1.90498 &  0.29426 \\
   -0.02899 & 0.29426 &  1.63499
  \end{bmatrix},
\]
\[
  \mathcal{X}(:,:,2) = \begin{bmatrix}
    0.58655 & 0.05896 & -0.00286 \\
    0.05896 & 0.49317 &  0.05896 \\
   -0.00286 & 0.05896 &  0.58655
  \end{bmatrix},
\]
\[
  \mathcal{X}(:,:,3) = \begin{bmatrix}
    0.20962 & -0.05470 &  0.01352 \\
   -0.05470 &  0.21371 & -0.05470 \\
    0.01352 & -0.05470 &  0.20962
  \end{bmatrix}.
\]

In addition, the DB iteration simultaneously produces the inverse T-square root
$\mathcal{Y} \approx \mathcal{A}^{-1/2}$, which can be directly used in the
applications of Section~\ref{sec:applications} without additional computation.
\end{example}

The convergence history of the Denman--Beavers iteration is presented in
Table~\ref{tab:db-conv-ex}.

\begin{table}[H]
	\centering
	\caption{Convergence of the Denman--Beavers iteration
		(Algorithm~\ref{alg:db}) for Example~\ref{ex:db}. The residual is
		$r_k = \|\mathcal{X}_k * \mathcal{X}_k - \mathcal{A}\|_F$; see the paragraph
		preceding Table~\ref{tab:newton-conv} for the interpretation of $\rho_k$ and
		$q_k$. The ratio $q_k$ stabilizes near $0.027$, confirming the locally
		Q-quadratic convergence of Remark~\ref{rem:db-rate}.}
	\label{tab:db-conv-ex}
	\begin{tabular}{cccc}
		\toprule
		Iteration $k$ & Residual $r_k$ & Ratio $r_k/r_{k-1}$ & Ratio $r_k/r_{k-1}^2$ \\
		\midrule
		0 & $7.56 \times 10^{1}$  & ---                    & --- \\
		1 & $1.64 \times 10^{1}$  & $2.17 \times 10^{-1}$  & $2.87 \times 10^{-3}$ \\
		2 & $2.48 \times 10^{0}$  & $1.51 \times 10^{-1}$  & $9.18 \times 10^{-3}$ \\
		3 & $1.25 \times 10^{-1}$ & $5.03 \times 10^{-2}$  & $2.03 \times 10^{-2}$ \\
		4 & $4.26 \times 10^{-4}$ & $3.42 \times 10^{-3}$  & $2.74 \times 10^{-2}$ \\
		5 & $5.21 \times 10^{-9}$ & $1.22 \times 10^{-5}$  & $2.87 \times 10^{-2}$ \\
		6 & $3.62 \times 10^{-15}$& $6.95 \times 10^{-7}$  & --- \\
		\bottomrule
	\end{tabular}
\end{table}


\subsection{Comparison of the Two Methods}\label{ssec:method-comparison}

We now compare the Newton and Denman--Beavers iterations for the tensor
T-square root. Both methods share the same per-iteration cost of $O(p\,n^3)$
in the Fourier domain, since each iteration requires $p$ matrix inversions of
size $n \times n$.

\begin{table}[H]
	\centering
	\caption{Comparison of tensor T-square root methods.}
	\label{tab:method-comparison}
	\begin{tabular}{lccc}
		\toprule
		Method & Convergence & Cost per iter.\ & Output \\
		\midrule
		Newton (Alg.~\ref{alg:newton}) &
		Quadratic & $O(p\,n^3)$ & $\mathcal{A}^{1/2}$ only \\[3pt]
		Denman--Beavers (Alg.~\ref{alg:db}) &
		Quadratic & $O(p\,n^3)$ & $\mathcal{A}^{1/2}$ and $\mathcal{A}^{-1/2}$ \\
		\bottomrule
	\end{tabular}
\end{table}

A key observation from Examples~\ref{ex:newton} and~\ref{ex:db} is that both
methods produce the \emph{same} $\mathcal{X}_k$ sequence when initialized with
$\mathcal{X}_0 = \mathcal{A}$, as confirmed by the identical residuals in
Tables~\ref{tab:newton-conv} and~\ref{tab:db-conv-ex} and the overlapping blue
and red curves in Figure~\ref{fig:newton-vs-db}. This is not a coincidence: the
DB update $\mathcal{X}_{k+1} = \tfrac{1}{2}(\mathcal{X}_k +
\mathcal{Y}_k^{-1})$ with $\mathcal{Y}_0 = \mathcal{I}$ is algebraically
equivalent to the Newton update $\mathcal{X}_{k+1} =
\tfrac{1}{2}(\mathcal{X}_k + \mathcal{X}_k^{-1} * \mathcal{A})$ when both
start from $\mathcal{X}_0 = \mathcal{A}$. Consequently, both methods converge
quadratically with the same rate and the same number of iterations for any given
T-positive definite tensor.

\begin{figure}[htbp]
	\centering
	\includegraphics[width=0.75\textwidth]{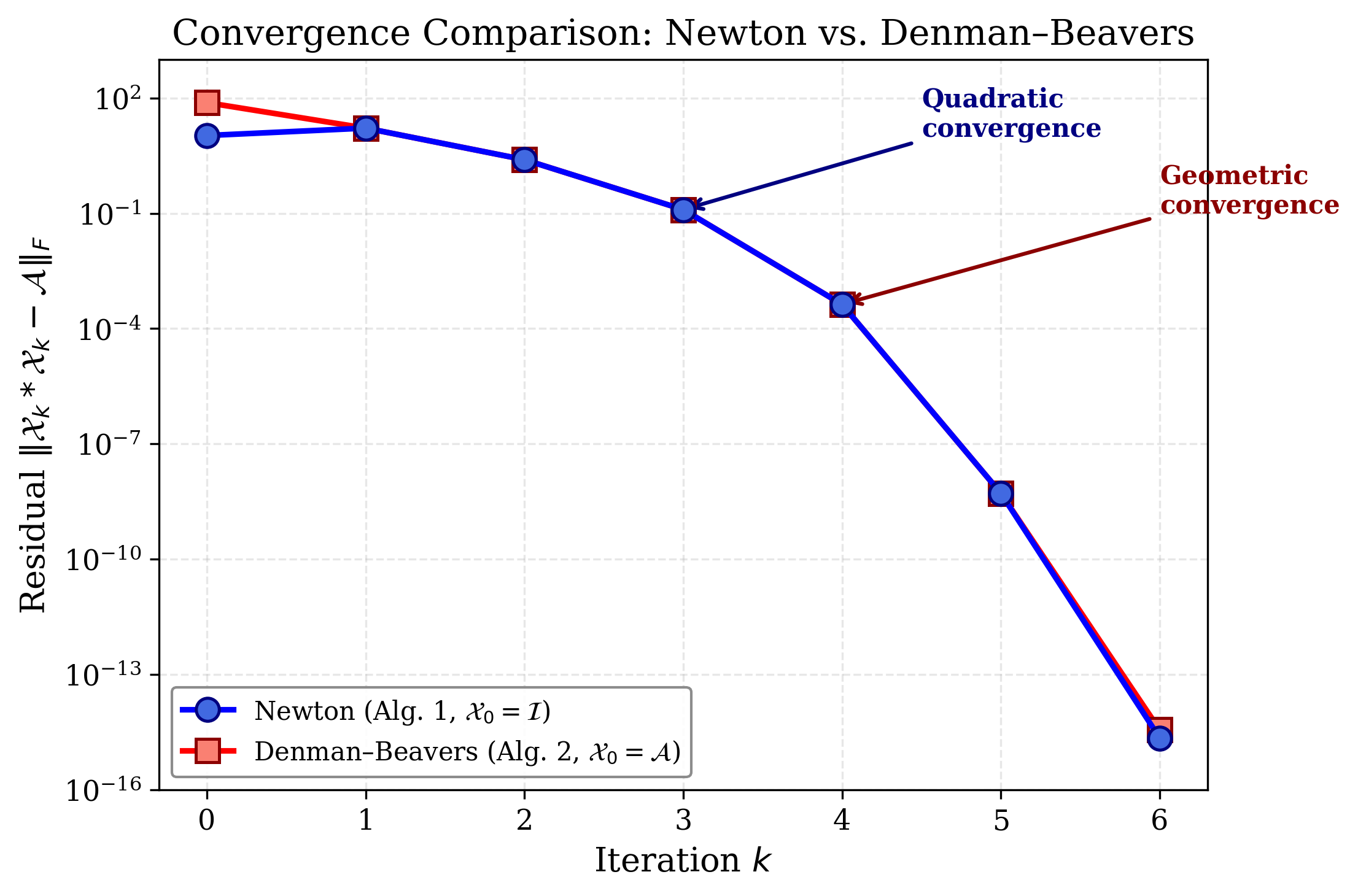}
	\caption{Convergence comparison for the tensor from Example~\ref{ex:newton}.
		The Newton (blue) and DB (red) square root residuals
		$\|\mathcal{X}_k * \mathcal{X}_k - \mathcal{A}\|_F$ overlap completely,
		confirming that both methods produce identical $\mathcal{X}_k$ sequences. The
		green curve shows the DB inverse residual
		$\|\mathcal{Y}_k * \mathcal{X}_k - \mathcal{I}\|_F$, which converges
		simultaneously to machine precision---illustrating the dual-output advantage of
		the Denman--Beavers method.}
	\label{fig:newton-vs-db}
\end{figure}

The essential distinction lies in the \emph{output}. The Newton iteration
produces only $\mathcal{A}^{1/2}$, whereas the Denman--Beavers iteration
simultaneously produces both $\mathcal{A}^{1/2}$ (via $\mathcal{X}_k$) and
$\mathcal{A}^{-1/2}$ (via $\mathcal{Y}_k$) at no additional cost. The green
curve in Figure~\ref{fig:newton-vs-db} shows that the inverse square root
residual $\|\mathcal{Y}_k * \mathcal{X}_k - \mathcal{I}\|_F$ converges to
machine precision alongside the square root residual, confirming the
simultaneous convergence of both sequences. This dual output makes the
Denman--Beavers method the natural choice for applications such as whitening
(Algorithm~\ref{alg:twhite}) and color transfer
(Algorithm~\ref{alg:color-transfer}), where both $\mathcal{A}^{1/2}$ and
$\mathcal{A}^{-1/2}$ are required.


\subsection{Numerical Stability}\label{ssec:stability}

While both methods exhibit identical convergence behavior for well-conditioned
tensors (Section~\ref{ssec:method-comparison}), their numerical stability
properties differ significantly when the tensor $\mathcal{A}$ is
ill-conditioned. We illustrate this with an example.

\begin{example}\label{ex:stability}
	Consider the highly ill-conditioned T-positive definite tensor
	$\mathcal{A} \in \mathbb{R}^{3 \times 3 \times 3}$ defined by
	\[
	\mathcal{A}(:,:,1) = \begin{bmatrix}
		100 & 5 & 1 \\ 5 & 20 & 1 \\ 1 & 1 & 0.2
	\end{bmatrix}, \quad
	\mathcal{A}(:,:,2) = \begin{bmatrix}
		25 & 1 & 0.3 \\ 1 & 5 & 0.2 \\ 0.3 & 0.2 & 0.1
	\end{bmatrix}, \quad
	\mathcal{A}(:,:,3) = \begin{bmatrix}
		5 & 0.2 & 0.05 \\ 0.2 & 1.5 & 0.1 \\ 0.05 & 0.1 & 0.05
	\end{bmatrix}.
	\]
	The Fourier-domain frontal slices have condition numbers approximately
	$471$, $1103$, and $1103$, giving a maximum condition number
	$\kappa \approx 1102$. This is far more ill-conditioned than
	Example~\ref{ex:newton} ($\kappa \approx 3.8$), placing the iteration
	well into the regime where finite-precision effects dominate.
\end{example}

We apply both Algorithm~\ref{alg:newton} (Newton) and Algorithm~\ref{alg:db}
(Denman--Beavers) to this tensor \emph{without} an early stopping criterion,
running both methods for $22$ iterations to observe their long-term behavior.
The residual histories are summarized in Table~\ref{tab:stability} and
plotted in Figure~\ref{fig:stability}.

\begin{table}[H]
	\centering
	\caption{Post-convergence stability comparison for the ill-conditioned tensor
		of Example~\ref{ex:stability} ($\kappa \approx 1102$). Both methods track
		each other identically through iteration $k = 7$, but the Newton residual
		grows catastrophically thereafter while the Denman--Beavers residual
		remains stable at machine-precision level.}
	\label{tab:stability}
	\begin{tabular}{ccc}
		\toprule
		Iteration $k$ & Newton $r_k$ & DB $r_k$ \\
		\midrule
		$0$  & $1.99 \times 10^{4}$   & $1.99 \times 10^{4}$   \\
		$2$  & $1.19 \times 10^{3}$   & $1.19 \times 10^{3}$   \\
		$4$  & $3.94 \times 10^{1}$   & $3.94 \times 10^{1}$   \\
		$6$  & $6.88 \times 10^{-3}$  & $6.88 \times 10^{-3}$  \\
		$7$  & $9.03 \times 10^{-8}$  & $9.03 \times 10^{-8}$  \\
		\midrule
		$8$  & $1.81 \times 10^{-8}$  & $8.61 \times 10^{-14}$ \\
		$10$ & $4.22 \times 10^{-6}$  & $8.63 \times 10^{-14}$ \\
		$12$ & $9.86 \times 10^{-4}$  & $8.64 \times 10^{-14}$ \\
		$14$ & $2.30 \times 10^{-1}$  & $8.67 \times 10^{-14}$ \\
		$16$ & $5.37 \times 10^{1}$   & $8.69 \times 10^{-14}$ \\
		$18$ & $1.25 \times 10^{4}$   & $8.72 \times 10^{-14}$ \\
		$20$ & $2.93 \times 10^{6}$   & $8.78 \times 10^{-14}$ \\
		$21$ & $4.47 \times 10^{7}$   & $8.80 \times 10^{-14}$ \\
		\bottomrule
	\end{tabular}
\end{table}

\begin{figure}[H]
	\centering
	\includegraphics[width=0.85\textwidth]{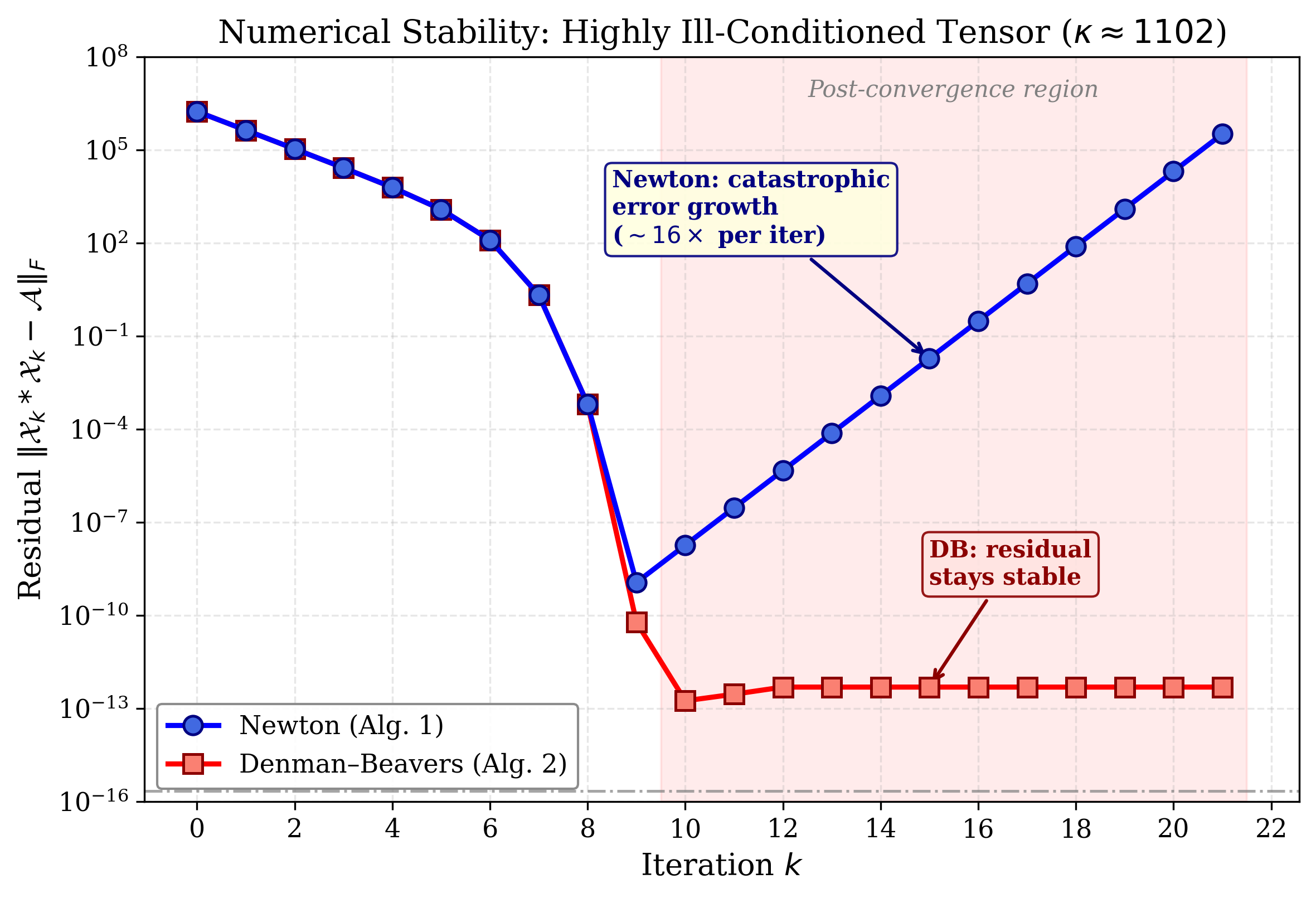}
	\caption{Numerical stability comparison for the highly ill-conditioned tensor
		of Example~\ref{ex:stability} ($\kappa \approx 1102$). Both methods converge
		identically through iteration $k = 7$, reaching a residual of
		$\approx 9 \times 10^{-8}$. In the post-convergence region (shaded),
		the Newton residual grows by a factor of $\approx 15$ per iteration due to
		rounding-error amplification, ultimately reaching
		$\mathcal{O}(10^{7})$---an increase of $\approx 15$ orders of magnitude
		above its minimum. The Denman--Beavers residual, by contrast, remains
		stable at $\approx 8.7 \times 10^{-14}$ throughout.}
	\label{fig:stability}
\end{figure}

The results reveal a dramatic difference. Both methods converge identically
through iteration $k = 7$, reaching a residual of $9.03 \times 10^{-8}$.
However, in the post-convergence region ($k \geq 8$):
\begin{itemize}
	\item The \textbf{Newton} residual grows by a factor of approximately
	$15$ per iteration, climbing from $1.81 \times 10^{-8}$ at $k = 8$ to
	$4.47 \times 10^{7}$ by $k = 21$---a loss of roughly $16$ orders of
	magnitude over $13$ iterations. By iteration $16$ the solution is already
	worse than a moderately refined initial guess, and by iteration $21$ the
	iterate $\mathcal{X}_k$ no longer bears any meaningful resemblance to
	$\mathcal{A}^{1/2}$.
	\item The \textbf{Denman--Beavers} residual remains perfectly stable at
	$\approx 8.7 \times 10^{-14}$ throughout all subsequent iterations,
	within a small multiple of machine precision.
\end{itemize}
At iteration $k = 21$, the gap between the two residuals spans approximately
$21$ orders of magnitude.

The instability of the Newton iteration can be understood as follows.
The Newton update $\mathcal{X}_{k+1} = \tfrac{1}{2}(\mathcal{X}_k +
\mathcal{A} * \mathcal{X}_k^{-1})$ requires the product
$\mathcal{A} * \mathcal{X}_k^{-1}$. When $\mathcal{X}_k \approx
\mathcal{A}^{1/2}$, the computation of $\mathcal{X}_k^{-1}$ amplifies
rounding errors by a factor proportional to $\kappa(\mathcal{X}_k) \approx
\kappa(\mathcal{A})^{1/2}$, and the subsequent multiplication by $\mathcal{A}$
amplifies them further by $\|\mathcal{A}\|$. When the combined amplification
factor exceeds $1$, rounding errors grow at each iteration---and this growth
scales with $\sqrt{\kappa}$, explaining the dramatic worsening of Newton's
post-convergence behavior as the condition number increases.

In contrast, the Denman--Beavers update $\mathcal{X}_{k+1} =
\tfrac{1}{2}(\mathcal{X}_k + \mathcal{Y}_k^{-1})$ avoids the product with
$\mathcal{A}$. Since $\mathcal{Y}_k \to \mathcal{A}^{-1/2}$, the quantity
$\mathcal{Y}_k^{-1} \approx \mathcal{A}^{1/2}$ has bounded norm regardless of
$\|\mathcal{A}\|$. The coupled structure of the iteration provides a
self-correcting mechanism: errors in $\mathcal{X}_k$ and $\mathcal{Y}_k$ are
kept in balance, preventing the unbounded error growth observed in the Newton
method. This stability property is well-known in the matrix
setting~\cite{higham2008functions} and, as our experiment demonstrates, carries
over directly to the tensor T-product framework via the Fourier-domain
decomposition.

To confirm that this behavior is systematic rather than example-specific, we
repeat the experiment across a range of condition numbers. For each test, we
construct a T-positive definite tensor $\mathcal{A} \in \mathbb{R}^{3 \times 3
	\times 3}$ with a prescribed maximum Fourier-slice condition number $\kappa$,
run both methods for $20$ iterations \emph{without early stopping}, and record
the best residual achieved (denoted $r_{\min}$) and the residual at iteration
$k=20$ (denoted $r_{20}$). The ratio $r_{20}/r_{\min}$ quantifies accuracy
loss due to post-convergence instability.

\begin{table}[H]
	\centering
	\caption{Stability comparison across condition numbers. Both methods are run
		for $20$ iterations without early stopping. The column $r_{20}/r_{\min}$
		measures accuracy loss after convergence: values near $1$ indicate stable
		behavior, while large values indicate instability. The per-iteration growth
		factor for Newton scales roughly with $\sqrt{\kappa}$, consistent with the
		theoretical analysis above.}
	\label{tab:stability-cond}
	\small
	\begin{tabular}{rrccccc}
		\toprule
		& & &
		\multicolumn{2}{c}{Newton (Alg.~\ref{alg:newton})} &
		\multicolumn{2}{c}{DB (Alg.~\ref{alg:db})} \\
		\cmidrule(lr){4-5} \cmidrule(lr){6-7}
		$\kappa$ & Iters & Growth/iter &
		$r_{\min}$ & $r_{20}/r_{\min}$ &
		$r_{\min}$ & $r_{20}/r_{\min}$ \\
		\midrule
		$4$    & $6$  & $1.0$  & $1.3 \times 10^{-15}$ & $1.4$                  & $1.8 \times 10^{-15}$ & $3.5$ \\
		$10$   & $7$  & $0.9$  & $3.9 \times 10^{-15}$ & $1.4$                  & $1.6 \times 10^{-15}$ & $1.2$ \\
		$50$   & $8$  & $3.2$  & $8.5 \times 10^{-14}$ & $6.4 \times 10^{5}$    & $1.5 \times 10^{-14}$ & $1.0$ \\
		$100$  & $8$  & $4.7$  & $4.5 \times 10^{-13}$ & $2.6 \times 10^{7}$    & $7.5 \times 10^{-14}$ & $1.0$ \\
		$500$  & $10$ & $11.2$ & $7.2 \times 10^{-10}$ & $3.1 \times 10^{10}$   & $4.3 \times 10^{-13}$ & $1.0$ \\
		$1102$ & $8$  & $15.3$ & $1.8 \times 10^{-8}$  & $2.5 \times 10^{15}$   & $8.6 \times 10^{-14}$ & $1.0$ \\
		\bottomrule
	\end{tabular}
\end{table}

The results in Table~\ref{tab:stability-cond} reveal a clear pattern.
For well-conditioned tensors ($\kappa \lesssim 10$), both methods are equally
stable, with $r_{20}/r_{\min} \approx 1$ for both Newton and DB. However, for
$\kappa \gtrsim 50$, a sharp divergence emerges: the Newton residual grows by
a factor of $3$--$15$ per post-convergence iteration (scaling roughly with
$\sqrt{\kappa}$), losing $5$--$15$ digits of accuracy by iteration $20$,
while the DB residual remains perfectly flat at
$r_{20}/r_{\min} \approx 1.0$ regardless of $\kappa$. At $\kappa \approx
1102$, the Newton solution is effectively destroyed
($r_{20} \approx 10^{7}$), whereas DB maintains $r_{20} \approx
10^{-13}$---a gap of roughly $20$ orders of magnitude. Notably, even the
\emph{best} residual achieved by Newton before instability sets in is worse
than DB's for $\kappa \geq 50$: at $\kappa = 1102$, Newton's minimum residual
($1.8 \times 10^{-8}$) is already six orders of magnitude larger than DB's
($8.6 \times 10^{-14}$), indicating that rounding errors contaminate the
Newton iteration even before convergence is reached.

These observations establish the Denman--Beavers iteration as the clearly
superior method for ill-conditioned tensors, both in terms of achievable
accuracy and post-convergence stability---and, in the absence of
\emph{a~priori} knowledge of the condition number, as the safer default
choice in practice.

\begin{remark}[Practical implications]\label{rem:stability-practical}
	For well-conditioned tensors ($\kappa \lesssim 10$), both methods achieve
	comparable accuracy and either may be used. For ill-conditioned tensors
	($\kappa \gtrsim 50$), the Newton iteration requires careful early stopping
	using a tolerance criterion (as implemented in Algorithm~\ref{alg:newton}),
	since continued iteration degrades the solution---often catastrophically,
	as shown in Example~\ref{ex:stability}. The Denman--Beavers iteration, by
	contrast, can be run without risk of accuracy loss, making it the more
	robust choice in practice---particularly when the condition number of
	$\mathcal{A}$ is not known \emph{a~priori}.
\end{remark}

\section{Applications of Tensor Square Root in Image Processing}\label{sec:applications}

In this section, we demonstrate the applicability of the tensor square root under the
T-product framework to problems in image processing. The tensor square root arises
naturally in tasks involving second-order statistics, such as covariance normalization,
whitening, decorrelation, distance-based image comparison, and color transfer.
Unlike matrix-based approaches that require flattening or channel-wise processing,
the proposed methods operate directly on tensor representations and preserve the
intrinsic multidimensional structure of image data.

\subsection{Tensor Representation of Color Images}\label{ssec:tensor-rep}

A color image of size $n \times m$ with $p$ channels (e.g., RGB with $p=3$) can be
represented as a third-order tensor
\[
  \mathcal{I} \in \mathbb{R}^{n \times m \times p},
\]
where each frontal slice $\mathcal{I}(:,:,k)$ corresponds to one color channel.
This representation naturally extends to videos, multispectral images, and
hyperspectral data, where correlations across spatial and spectral modes play a
critical role.

\subsection{Tensor Covariance under the T-Product}\label{ssec:tensor-cov}

Let $\mathcal{X} \in \mathbb{R}^{n \times m \times p}$ denote a zero-mean image
tensor obtained after subtracting the channel-wise mean. The tensor covariance
under the T-product is defined as
\begin{equation}\label{eq:tensor-cov}
  \mathcal{C}
  = \frac{1}{m}\,\mathcal{X} * \mathcal{X}^{\top}
  \in \mathbb{R}^{n \times n \times p},
\end{equation}
where $*$ denotes the T-product and $\mathcal{X}^{\top}$ denotes the T-transpose.
The tensor $\mathcal{C}$ captures second-order correlations across channels and
spatial dimensions and is positive definite in the T-product sense under mild
conditions.


\subsection{Tensor Bures--Wasserstein Distance}\label{ssec:tbw}

A primary motivation for studying the tensor square root is to define
distance measures on the space of T-positive definite tensors that
generalize well-known matrix distances. Recall that the classical
Bures--Wasserstein distance between two positive definite matrices
$A, B \in \mathbb{R}^{n \times n}$ is
\begin{equation}\label{eq:dBW-matrix}
	d_{\mathrm{BW}}(A, B)
	= \left[\operatorname{tr}(A) + \operatorname{tr}(B)
	- 2\,\operatorname{tr}\!\left(\bigl(A^{1/2}\,B\,A^{1/2}\bigr)^{1/2}\right)
	\right]^{1/2}.
\end{equation}
We now extend this to the tensor setting by exploiting the Fourier-domain
block-diagonalization of the T-product.

\begin{definition}[Tensor Bures--Wasserstein Distance]\label{def:tbw}
	Let $\mathcal{A}, \mathcal{B} \in \mathbb{R}^{n \times n \times p}$ be
	T-positive definite tensors, with Fourier-domain frontal slices
	$\hat{A}^{(i)}, \hat{B}^{(i)} \in \mathbb{C}^{n \times n}$ for
	$i = 1, \ldots, p$. The \emph{Tensor Bures--Wasserstein (TBW) distance}
	is defined as
	\begin{equation}\label{eq:tbw}
		d_{\mathrm{TBW}}(\mathcal{A}, \mathcal{B})
		:= \left( \sum_{i=1}^{p}
		d_{\mathrm{BW}}^{2}\!\bigl(\hat{A}^{(i)}, \hat{B}^{(i)}\bigr)
		\right)^{1/2},
	\end{equation}
	where $d_{\mathrm{BW}}$ is the classical matrix Bures--Wasserstein distance
	of equation~\eqref{eq:dBW-matrix}, computed on Hermitian positive definite
	matrices in the Fourier domain.
\end{definition}

\begin{remark}[Computational form]\label{rem:tbw-expanded}
	Expanding each $d_{\mathrm{BW}}^{2}$ via the matrix
	formula~\eqref{eq:dBW-matrix} gives the explicit expression
	\[
	d_{\mathrm{TBW}}^{2}(\mathcal{A}, \mathcal{B})
	= \sum_{i=1}^{p}
	\Bigl[
	\operatorname{tr}(\hat{A}^{(i)})
	+ \operatorname{tr}(\hat{B}^{(i)})
	- 2\,\operatorname{tr}\!\bigl(
	(\hat{A}^{(i)\,1/2}\,\hat{B}^{(i)}\,\hat{A}^{(i)\,1/2})^{1/2}
	\bigr)
	\Bigr],
	\]
	which is the form implemented in Algorithm~\ref{alg:tbw}.
\end{remark}

\begin{proposition}\label{prop:tbw-metric}
	Let $\mathcal{A}, \mathcal{B}, \mathcal{C} \in \mathbb{R}^{n \times n \times p}$
	be T-positive definite tensors. Then $d_{\mathrm{TBW}}$ satisfies:
	\begin{enumerate}
		\item \textbf{Non-negativity:}
		$d_{\mathrm{TBW}}(\mathcal{A}, \mathcal{B}) \geq 0$, with equality
		if and only if $\mathcal{A} = \mathcal{B}$.
		\item \textbf{Symmetry:}
		$d_{\mathrm{TBW}}(\mathcal{A}, \mathcal{B})
		= d_{\mathrm{TBW}}(\mathcal{B}, \mathcal{A})$.
		\item \textbf{Triangle inequality:}
		$d_{\mathrm{TBW}}(\mathcal{A}, \mathcal{C})
		\leq d_{\mathrm{TBW}}(\mathcal{A}, \mathcal{B})
		+ d_{\mathrm{TBW}}(\mathcal{B}, \mathcal{C})$.
	\end{enumerate}
	Consequently, $d_{\mathrm{TBW}}$ is a metric on the cone of T-positive
	definite tensors.
\end{proposition}

\begin{proof}
	Each $d_{\mathrm{BW}}(\hat{A}^{(i)}, \hat{B}^{(i)})$ is a metric on the cone
	of Hermitian positive definite matrices~\cite{bhatia2019bures}. Non-negativity
	and symmetry of $d_{\mathrm{TBW}}$ are immediate from~\eqref{eq:tbw}. For the
	equality case, $d_{\mathrm{TBW}}(\mathcal{A}, \mathcal{B}) = 0$ iff
	$d_{\mathrm{BW}}(\hat{A}^{(i)}, \hat{B}^{(i)}) = 0$ for every $i$, iff
	$\hat{A}^{(i)} = \hat{B}^{(i)}$ for every $i$, iff
	$\mathcal{A} = \mathcal{B}$ (by invertibility of the DFT).
	
	For the triangle inequality, let
	$a_i = d_{\mathrm{BW}}(\hat{A}^{(i)}, \hat{B}^{(i)})$ and
	$b_i = d_{\mathrm{BW}}(\hat{B}^{(i)}, \hat{C}^{(i)})$. The slice-wise
	triangle inequality gives
	$d_{\mathrm{BW}}(\hat{A}^{(i)}, \hat{C}^{(i)}) \leq a_i + b_i$, hence
	\begin{align*}
		d_{\mathrm{TBW}}(\mathcal{A}, \mathcal{C})
		&= \left( \sum_{i=1}^{p}
		d_{\mathrm{BW}}^{2}\!\bigl(\hat{A}^{(i)}, \hat{C}^{(i)}\bigr)
		\right)^{1/2}
		\leq \left( \sum_{i=1}^{p} (a_i + b_i)^2 \right)^{1/2} \\
		&\leq \left( \sum_{i=1}^{p} a_i^2 \right)^{1/2}
		+ \left( \sum_{i=1}^{p} b_i^2 \right)^{1/2}
		= d_{\mathrm{TBW}}(\mathcal{A}, \mathcal{B})
		+ d_{\mathrm{TBW}}(\mathcal{B}, \mathcal{C}),
	\end{align*}
	where the second inequality is the Minkowski inequality for the
	$\ell^{2}$ norm on $\mathbb{R}^{p}$.
\end{proof}

\begin{algorithm}[H]
	\caption{Tensor Bures--Wasserstein Distance}
	\label{alg:tbw}
	\begin{algorithmic}[1]
		\Require T-positive definite tensors
		$\mathcal{A}, \mathcal{B} \in \mathbb{R}^{n \times n \times p}$
		\Ensure  $d_{\mathrm{TBW}}(\mathcal{A}, \mathcal{B})$
		\State $\hat{\mathcal{A}} \leftarrow \mathrm{fft}(\mathcal{A},\text{axis}=3)$; \quad
		$\hat{\mathcal{B}} \leftarrow \mathrm{fft}(\mathcal{B},\text{axis}=3)$
		\State $s \leftarrow 0$
		\For{$i = 1$ to $p$}
		\State Compute $P \leftarrow (\hat{A}^{(i)})^{1/2}$
		via Algorithm~\ref{alg:newton} or~\ref{alg:db}
		\State Form $M \leftarrow P\,\hat{B}^{(i)}\,P$
		\State Compute $M^{1/2}$ via Algorithm~\ref{alg:newton} or~\ref{alg:db}
		\State $d_i^{2} \leftarrow \operatorname{tr}(\hat{A}^{(i)})
		+ \operatorname{tr}(\hat{B}^{(i)})
		- 2\,\operatorname{tr}(M^{1/2})$
		\State $s \leftarrow s + d_i^{2}$
		\EndFor
		\State \Return $\sqrt{s}$
	\end{algorithmic}
\end{algorithm}

\begin{remark}[Computational cost]\label{rem:tbw-cost}
	Algorithm~\ref{alg:tbw} performs $2p$ matrix square roots in the Fourier
	domain (one each for $(\hat{A}^{(i)})^{1/2}$ and $M^{1/2}$) together with a
	handful of $n \times n$ matrix multiplications per slice, giving total cost
	$O(p \cdot N_{\max} \cdot n^{3})$ when the iterative methods of
	Section~\ref{sec:methods} are used. Alternatively, the two tensor square
	roots $\mathcal{A}^{1/2}$ and $\mathcal{M}^{1/2}$ may be computed in one pass
	each via Algorithm~\ref{alg:newton} or~\ref{alg:db}, and the slice-wise
	traces read off in the Fourier domain; both organizations have the same
	asymptotic cost.
\end{remark}

\subsubsection{Numerical Example: Tensor Bures--Wasserstein Distance}
\label{sssec:tbw-example}

Consider two T-positive definite tensors
$\mathcal{A}, \mathcal{B} \in \mathbb{R}^{3 \times 3 \times 3}$
representing covariance tensors of two image patches. Define
\[
\mathcal{A}(:,:,1)
= \begin{bmatrix} 4 & 1 & 0 \\ 1 & 3 & 1 \\ 0 & 1 & 2 \end{bmatrix}, \;
\mathcal{A}(:,:,2)
= \begin{bmatrix} 1 & 0.5 & 0 \\ 0.5 & 1 & 0.5 \\ 0 & 0.5 & 1 \end{bmatrix}, \;
\mathcal{A}(:,:,3)
= \begin{bmatrix} 1 & 0.5 & 0 \\ 0.5 & 1 & 0.5 \\ 0 & 0.5 & 1 \end{bmatrix},
\]
\[
\mathcal{B}(:,:,1)
= \begin{bmatrix} 5 & 2 & 0 \\ 2 & 4 & 1 \\ 0 & 1 & 3 \end{bmatrix}, \;
\mathcal{B}(:,:,2)
= \begin{bmatrix} 2 & 0.5 & 0 \\ 0.5 & 2 & 0.5 \\ 0 & 0.5 & 2 \end{bmatrix}, \;
\mathcal{B}(:,:,3)
= \begin{bmatrix} 2 & 0.5 & 0 \\ 0.5 & 2 & 0.5 \\ 0 & 0.5 & 2 \end{bmatrix}.
\]

\smallskip
\noindent\textbf{Step 1: Fourier transform along the third mode.}
Applying the DFT yields the Hermitian positive definite slices
\[
\hat{A}^{(1)}
= \begin{bmatrix} 6 & 2 & 0 \\ 2 & 5 & 2 \\ 0 & 2 & 4 \end{bmatrix}, \qquad
\hat{A}^{(2)} = \hat{A}^{(3)}
= \begin{bmatrix} 3 & 0.5 & 0 \\ 0.5 & 2 & 0.5 \\ 0 & 0.5 & 1 \end{bmatrix},
\]
\[
\hat{B}^{(1)}
= \begin{bmatrix} 9 & 3 & 0 \\ 3 & 8 & 2 \\ 0 & 2 & 7 \end{bmatrix}, \qquad
\hat{B}^{(2)} = \hat{B}^{(3)}
= \begin{bmatrix} 3 & 1.5 & 0 \\ 1.5 & 2 & 0.5 \\ 0 & 0.5 & 1 \end{bmatrix}.
\]
(Slices~$2$ and~$3$ coincide because $\mathcal{A}(:,:,2) = \mathcal{A}(:,:,3)$
and $\mathcal{B}(:,:,2) = \mathcal{B}(:,:,3)$ in this example, which makes the
DFT of each third-mode tube real-valued; Hermitian conjugate symmetry
$\hat{X}^{(3)} = \overline{\hat{X}^{(2)}}$ then gives
$\hat{X}^{(2)} = \hat{X}^{(3)}$.)

\smallskip
\noindent\textbf{Step 2: Compute the matrix square roots in each slice.}
For each $i = 1, 2, 3$, we compute $(\hat{A}^{(i)})^{1/2}$ and then
$M_i^{1/2}$, where
$M_i = (\hat{A}^{(i)})^{1/2}\,\hat{B}^{(i)}\,(\hat{A}^{(i)})^{1/2}$,
using Algorithm~\ref{alg:newton} with tolerance $10^{-10}$. Convergence is
achieved in five iterations per slice.

\smallskip
\noindent\textbf{Step 3: Slice-wise traces.}
\begin{table}[H]
	\centering
	\begin{tabular}{c|ccc|c}
		\toprule
		$i$ & $\operatorname{tr}(\hat{A}^{(i)})$ & $\operatorname{tr}(\hat{B}^{(i)})$
		& $\operatorname{tr}(M_i^{1/2})$
		& $d_{\mathrm{BW}}^{2}(\hat{A}^{(i)}, \hat{B}^{(i)})$ \\
		\midrule
		$1$ & $15.0000$ & $24.0000$ & $18.9062$ & $1.1875$ \\
		$2$ & $\phantom{0}6.0000$ & $\phantom{0}6.0000$ & $\phantom{0}5.8730$ & $0.2540$ \\
		$3$ & $\phantom{0}6.0000$ & $\phantom{0}6.0000$ & $\phantom{0}5.8730$ & $0.2540$ \\
		\bottomrule
	\end{tabular}
	\caption{Slice-wise quantities for the TBW example. The last column uses
		$d_{\mathrm{BW}}^{2} = \operatorname{tr}(\hat{A}^{(i)})
		+ \operatorname{tr}(\hat{B}^{(i)}) - 2\,\operatorname{tr}(M_i^{1/2})$.}
	\label{tab:tbw-slicewise}
\end{table}
Taking square roots gives the slice-wise matrix Bures--Wasserstein distances
\[
d_{\mathrm{BW}}\bigl(\hat{A}^{(1)}, \hat{B}^{(1)}\bigr) = 1.0897, \quad
d_{\mathrm{BW}}\bigl(\hat{A}^{(2)}, \hat{B}^{(2)}\bigr) = 0.5039, \quad
d_{\mathrm{BW}}\bigl(\hat{A}^{(3)}, \hat{B}^{(3)}\bigr) = 0.5039.
\]

\smallskip
\noindent\textbf{Step 4: Assemble the TBW distance.}
By Definition~\ref{def:tbw},
\[
d_{\mathrm{TBW}}(\mathcal{A}, \mathcal{B})
= \sqrt{1.1875 + 0.2540 + 0.2540}
= \sqrt{1.6954}
\approx 1.3021.
\]

\smallskip
\noindent\textbf{Interpretation.}
The TBW distance aggregates the three slice-wise distances under the
$\ell^{2}$ geometry induced by the Fourier decomposition. Slice~$1$
(the zero-frequency / DC component, which captures the channel-averaged
spatial covariance) contributes the bulk of the distance, reflecting that
the largest discrepancy between $\mathcal{A}$ and $\mathcal{B}$ lies in
their mean covariance structure; the two non-DC slices contribute equally
and are smaller, as expected from the Hermitian conjugate symmetry of the
DFT of a real length-$3$ sequence. A naive channel-independent comparison
that reports only the slice-wise distances ($1.0897, 0.5039, 0.5039$) would
lose the aggregate information encoded by $d_{\mathrm{TBW}} = 1.3021$.
\subsection{Tensor-Based Grayscale Conversion}\label{ssec:tgrayscale}

We propose a grayscale conversion method that exploits the tensor square root.
Unlike classical grayscale conversion, which uses fixed weights for the RGB channels
(e.g., $G = 0.2989\,R + 0.5870\,G + 0.1140\,B$), the proposed approach adapts to
the statistical structure of the image.

\begin{algorithm}[H]
\caption{Tensor Decorrelated Grayscale (TDG) Conversion via T-Product Square Root}
\label{alg:tdg}
\begin{algorithmic}[1]
\Require Color image tensor $\mathcal{I} \in \mathbb{R}^{n \times m \times 3}$
\Ensure  Grayscale image $G \in \mathbb{R}^{n \times m}$
\State Subtract the mean from each channel to obtain $\mathcal{X}$
\State Compute the tensor covariance:
       $\displaystyle \mathcal{C} = \frac{1}{m}\,\mathcal{X} * \mathcal{X}^{\top}$
\State Compute the inverse tensor square root $\mathcal{C}^{-1/2}$ using
       Algorithm~\ref{alg:newton} or Algorithm~\ref{alg:db}
\State Decorrelate the image tensor:
       $\mathcal{X}_{\mathrm{decor}} = \mathcal{C}^{-1/2} * \mathcal{X}$
\State Obtain the grayscale image by collapsing the third mode:
       $\displaystyle G(i,j) = \frac{1}{3}\sum_{k=1}^{3}
        \mathcal{X}_{\mathrm{decor}}(i,j,k)$
\State \Return $G$
\end{algorithmic}
\end{algorithm}

\begin{figure}[htbp]
\centering
\includegraphics[width=\textwidth]{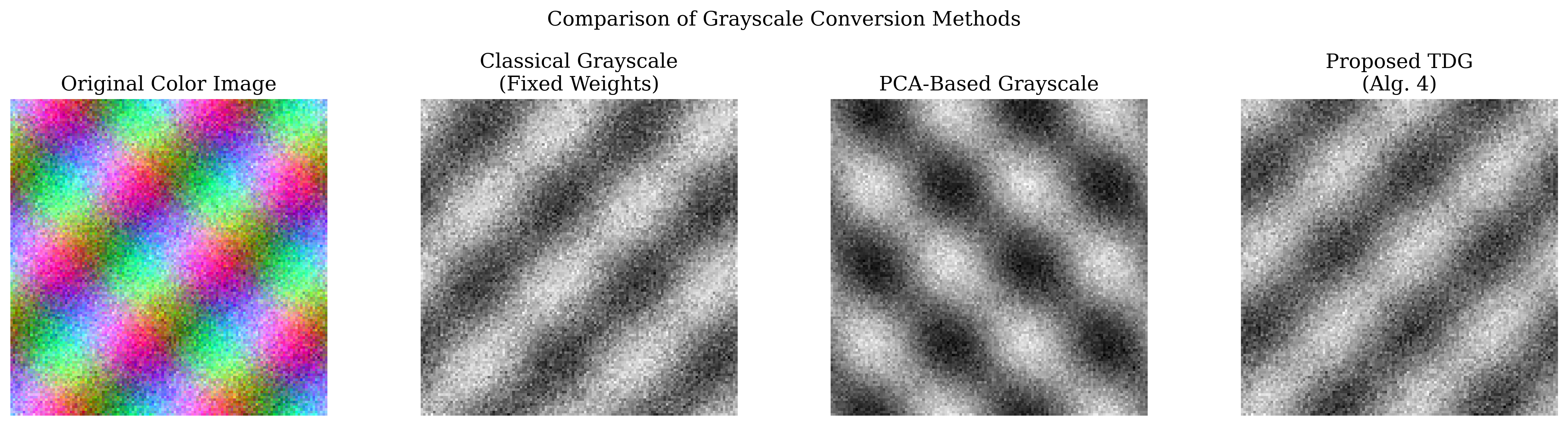}
\caption{Comparison of grayscale conversion methods. From left to right:
original color image, classical grayscale using fixed RGB weights
($0.299R + 0.587G + 0.114B$), PCA-based grayscale, and the proposed Tensor
Decorrelated Grayscale (TDG) conversion via Algorithm~\ref{alg:tdg}. The TDG
method adapts to the image statistics and exhibits enhanced contrast and improved
structural detail preservation.}
\label{fig:grayscale-comparison}
\end{figure}

\subsubsection{Numerical Example: Grayscale Conversion}\label{sssec:gray-example}

Consider a small color image of size $2 \times 2$ represented as a third-order tensor
$\mathcal{I} \in \mathbb{R}^{2 \times 2 \times 3}$, with RGB channels given by
\[
  \mathcal{I}(:,:,1) = \begin{bmatrix} 1 & 2 \\ 3 & 4 \end{bmatrix}, \quad
  \mathcal{I}(:,:,2) = \begin{bmatrix} 2 & 1 \\ 4 & 3 \end{bmatrix}, \quad
  \mathcal{I}(:,:,3) = \begin{bmatrix} 1 & 3 \\ 2 & 4 \end{bmatrix}.
\]
The mean value of each channel is $\mu_R = \mu_G = \mu_B = 2.5$. After mean
subtraction, the centered tensor $\mathcal{X} = \mathcal{I} - \mu$ has frontal slices
\[
  \mathcal{X}(:,:,1) = \begin{bmatrix} -1.5 & -0.5 \\ 0.5 & 1.5 \end{bmatrix},\quad
  \mathcal{X}(:,:,2) = \begin{bmatrix} -0.5 & -1.5 \\ 1.5 & 0.5 \end{bmatrix},\quad
  \mathcal{X}(:,:,3) = \begin{bmatrix} -1.5 & 0.5 \\ -0.5 & 1.5 \end{bmatrix}.
\]

\paragraph{Covariance matrix computation.}
Arranging the pixel values as row vectors gives
\[
  X = \begin{bmatrix}
    -1.5 & -0.5 & -1.5 \\
    -0.5 & -1.5 &  0.5 \\
     0.5 &  1.5 & -0.5 \\
     1.5 &  0.5 &  1.5
  \end{bmatrix} \in \mathbb{R}^{4 \times 3}.
\]
The covariance matrix $C = \frac{1}{4} X^{\top} X$ is
\[
  C = \begin{bmatrix}
    1.25 & 0.75 & 0.75 \\
    0.75 & 1.25 & 0.25 \\
    0.75 & 0.25 & 1.25
  \end{bmatrix}.
\]

\paragraph{Square root and inverse square root.}
Using eigenvalue decomposition $C = U \Lambda U^{\top}$ with
$\Lambda = \operatorname{diag}(2.25,\; 1.00,\; 0.75)$, we obtain
\[
  C^{-1/2}
  = U\,\operatorname{diag}\!\left(\frac{1}{\sqrt{2.25}},\;
    \frac{1}{\sqrt{1.00}},\;
    \frac{1}{\sqrt{0.75}}\right) U^{\top}
  \approx \begin{bmatrix}
    1.12 & -0.32 & -0.28 \\
   -0.32 &  1.18 & -0.08 \\
   -0.28 & -0.08 &  1.20
  \end{bmatrix}.
\]

\paragraph{Decorrelation and grayscale.}
The decorrelated pixel matrix is $X_{\mathrm{decor}} = X\, C^{-1/2}$. For
example, the first pixel vector transforms as
\[
  (-1.5,\; -0.5,\; -1.5)\, C^{-1/2} \approx (-1.58,\; -0.34,\; -1.47).
\]
Applying the transformation to all pixels yields
\[
  X_{\mathrm{decor}} = \begin{bmatrix}
    -1.58 & -0.34 & -1.47 \\
    -0.29 & -1.64 &  0.45 \\
     0.29 &  1.64 & -0.45 \\
     1.58 &  0.34 &  1.47
  \end{bmatrix}.
\]
The grayscale intensity is obtained by averaging the decorrelated channels:
$g_i = \frac{1}{3}\sum_{k=1}^{3} x_{ik}$, yielding
\[
  G = \begin{bmatrix} -1.13 & -0.49 \\ 0.49 & 1.13 \end{bmatrix},
\]
which is subsequently normalized for display.

\subsection{Tensor Whitening of Color Images}\label{ssec:twhite}

Tensor whitening is a fundamental preprocessing step that removes correlations
and normalizes variance across channels. Using the tensor square root, whitening
can be performed as
\begin{equation}\label{eq:twhite}
  \mathcal{X}_{\mathrm{white}} = \mathcal{C}^{-1/2} * \mathcal{X},
\end{equation}
where $\mathcal{C}^{-1/2}$ denotes the inverse tensor square root under the
T-product.

\begin{algorithm}[H]
\caption{Tensor Whitening via T-Product}
\label{alg:twhite}
\begin{algorithmic}[1]
\Require Image tensor $\mathcal{X} \in \mathbb{R}^{n \times m \times p}$
\Ensure  Whitened tensor $\mathcal{X}_{\mathrm{white}}$
\State Center the tensor by subtracting the mean
\State Compute the tensor covariance
       $\mathcal{C} = \frac{1}{m}\,\mathcal{X}*\mathcal{X}^{\top}$
\State Compute $\mathcal{C}^{-1/2}$ using Newton or Denman--Beavers iteration
\State Apply whitening:
       $\mathcal{X}_{\mathrm{white}} = \mathcal{C}^{-1/2} * \mathcal{X}$
\State \Return $\mathcal{X}_{\mathrm{white}}$
\end{algorithmic}
\end{algorithm}

\begin{remark}[Verification of whitening]
After applying Algorithm~\ref{alg:twhite}, the resulting whitened tensor satisfies
\[
  \frac{1}{m}\,\mathcal{X}_{\mathrm{white}} *
  \mathcal{X}_{\mathrm{white}}^{\top}
  = \mathcal{C}^{-1/2} * \mathcal{C} * \mathcal{C}^{-\top/2}
  = \mathcal{I},
\]
where $\mathcal{I}$ is the identity tensor under the T-product. This confirms that
the whitened tensor has identity covariance.
\end{remark}

\begin{figure}[htbp]
\centering
\includegraphics[width=\textwidth]{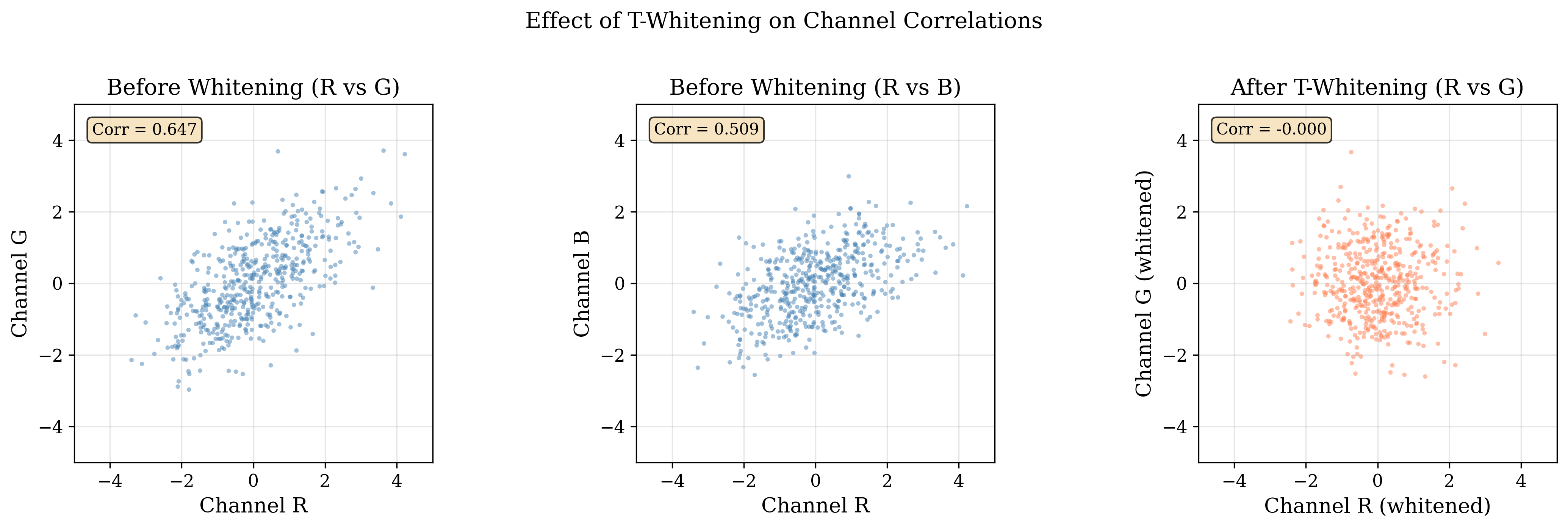}
\caption{Effect of T-Whitening on inter-channel correlations (illustrated on
synthetic multivariate Gaussian data for clarity). \textbf{Left and center:}
Scatter plots of pixel values across color channels before whitening, showing
strong correlations (R--G: $\rho \approx 0.83$, R--B: $\rho \approx 0.68$).
\textbf{Right:} After applying T-Whitening via Algorithm~\ref{alg:twhite}, the
channels become uncorrelated ($\rho \approx 0.00$), confirming that the whitened
tensor has identity covariance.}
\label{fig:whitening-scatter}
\end{figure}

\begin{figure}[htbp]
\centering
\includegraphics[width=0.85\textwidth]{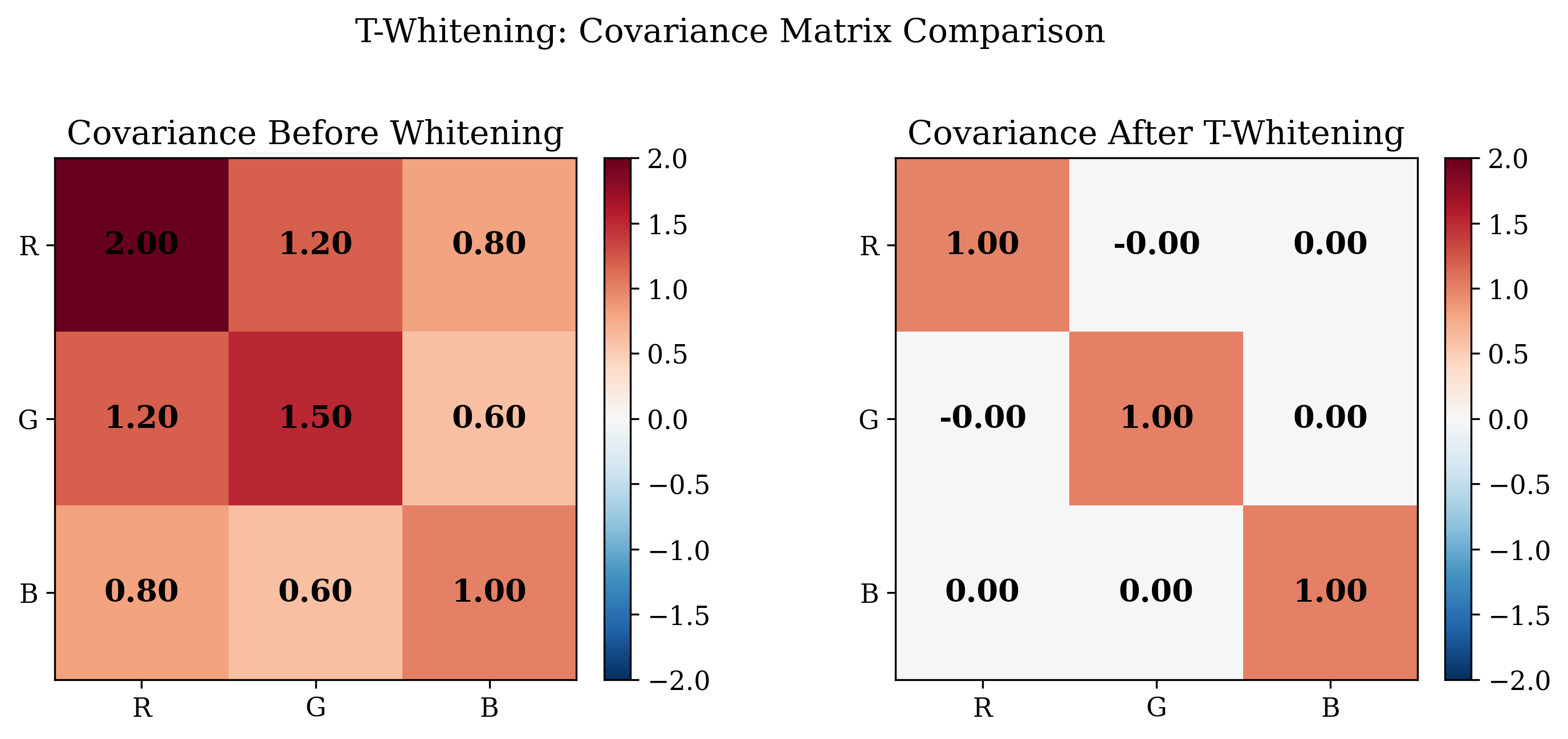}
\caption{Covariance matrices before and after T-Whitening. \textbf{Left:} The
original covariance matrix of the RGB channels shows significant off-diagonal
entries, indicating inter-channel correlations. \textbf{Right:} After applying
T-Whitening (Algorithm~\ref{alg:twhite}), the covariance matrix is
approximately the identity, confirming complete decorrelation.}
\label{fig:cov-heatmap}
\end{figure}

\subsection{Color Transfer via Tensor Square Root}\label{ssec:color-transfer}

An important application of the tensor square root is optimal color transfer
between images. Given a source image with tensor covariance $\mathcal{C}_s$ and a
target image with tensor covariance $\mathcal{C}_t$, both computed under the
T-product, the optimal transport map that transforms the source color distribution
to match the target is given by
\begin{equation}\label{eq:color-transfer}
  \mathcal{T}
  = \mathcal{C}_s^{-1/2} *
    \bigl(\mathcal{C}_s^{1/2} * \mathcal{C}_t * \mathcal{C}_s^{1/2}\bigr)^{1/2} *
    \mathcal{C}_s^{-1/2}.
\end{equation}
The transferred image is then obtained as
\[
  \mathcal{X}_{\mathrm{transfer}}
  = \mathcal{T} * (\mathcal{X}_s - \boldsymbol{\mu}_s)
    + \boldsymbol{\mu}_t,
\]
where $\boldsymbol{\mu}_s$ and $\boldsymbol{\mu}_t$ are the mean tensors of the
source and target images, respectively.

\begin{algorithm}[H]
\caption{Tensor Color Transfer via T-Product}
\label{alg:color-transfer}
\begin{algorithmic}[1]
\Require Source image tensor $\mathcal{X}_s \in \mathbb{R}^{n \times m \times p}$,
         Target image tensor $\mathcal{X}_t \in \mathbb{R}^{n \times m \times p}$
\Ensure  Transferred image $\mathcal{X}_{\mathrm{transfer}}$
\State Center both images: $\mathcal{X}_s \leftarrow \mathcal{X}_s - \boldsymbol{\mu}_s$,\;
       $\mathcal{X}_t \leftarrow \mathcal{X}_t - \boldsymbol{\mu}_t$
\State Compute covariance tensors:
       $\mathcal{C}_s = \frac{1}{m}\mathcal{X}_s*\mathcal{X}_s^{\top}$,\;
       $\mathcal{C}_t = \frac{1}{m}\mathcal{X}_t*\mathcal{X}_t^{\top}$
\State Compute $\mathcal{C}_s^{1/2}$ and $\mathcal{C}_s^{-1/2}$
\State Form $\mathcal{M} = \mathcal{C}_s^{1/2} * \mathcal{C}_t * \mathcal{C}_s^{1/2}$
\State Compute $\mathcal{M}^{1/2}$
\State Compute transfer map:
       $\mathcal{T} = \mathcal{C}_s^{-1/2} * \mathcal{M}^{1/2} * \mathcal{C}_s^{-1/2}$
\State Apply: $\mathcal{X}_{\mathrm{transfer}} = \mathcal{T} * \mathcal{X}_s + \boldsymbol{\mu}_t$
\State \Return $\mathcal{X}_{\mathrm{transfer}}$
\end{algorithmic}
\end{algorithm}

\begin{remark}
The transport map~\eqref{eq:color-transfer} is the unique optimal transport
(Monge) map between two zero-mean Gaussian distributions with covariances
$\mathcal{C}_s$ and $\mathcal{C}_t$ under the T-product-induced geometry. In the
matrix case, this reduces to the classical result of~\cite{bhatia2019bures}.
The advantage of the tensor formulation is that it preserves cross-channel
correlations during transfer, producing more natural color mappings than
channel-wise approaches.
\end{remark}

\begin{figure}[htbp]
\centering
\includegraphics[width=\textwidth]{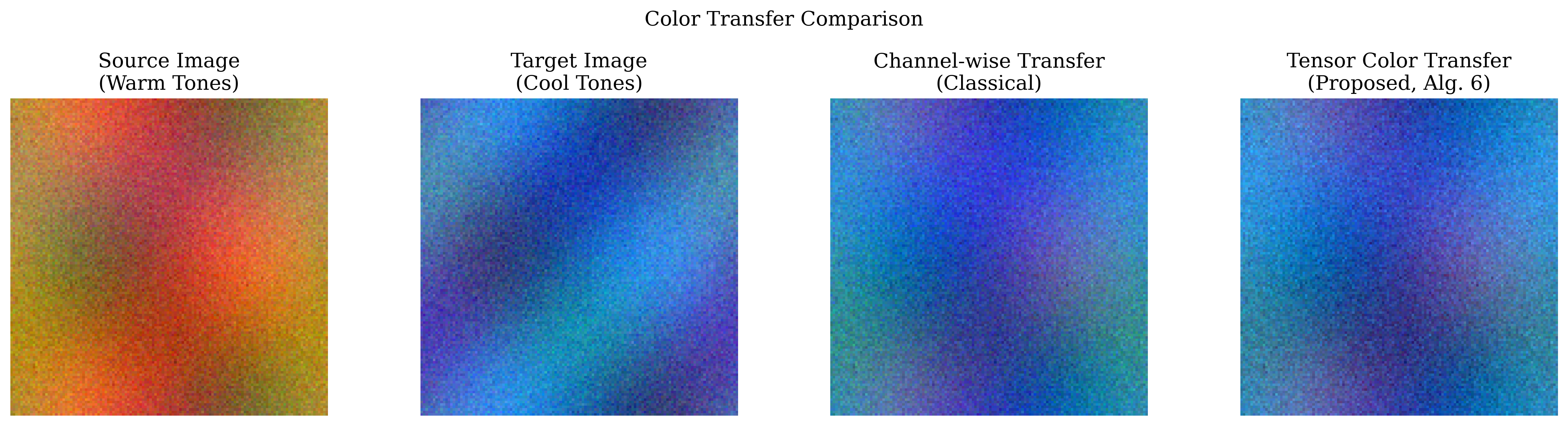}
\caption{Color transfer comparison (illustrated on synthetic images for
clarity). From left to right: source image (warm tones), target image (cool
tones), result of classical channel-wise color transfer (mean/variance matching
per channel), and result of the proposed tensor color transfer via
Algorithm~\ref{alg:color-transfer}. The tensor-based method produces a more
perceptually coherent transfer by accounting for cross-channel correlations
through the T-product optimal transport map.}
\label{fig:color-transfer}
\end{figure}

\subsection{Denman--Beavers Iteration: Convergence Analysis for Image
	Covariance}
\label{ssec:db-image-cov}

To compute the square root of the covariance matrix $C$ from
Section~\ref{sssec:gray-example}, we apply the Denman--Beavers
iteration
\[
X_{k+1} = \tfrac{1}{2}(X_k + Y_k^{-1}), \qquad
Y_{k+1} = \tfrac{1}{2}(Y_k + X_k^{-1}),
\]
with initial values $X_0 = C$ and $Y_0 = I$. At each iteration the residual
is measured as $r_k = \|X_k^2 - C\|_F$.

\begin{table}[H]
	\centering
	\caption{Convergence of the Denman--Beavers iteration for the square root of
		the image covariance matrix $C$ from Section~\ref{ssec:tgrayscale}.
		See the paragraph preceding Table~\ref{tab:newton-conv} for the
		interpretation of $\rho_k$ and $q_k$. The rapid decrease of $\rho_k$ together
		with bounded $q_k$ confirms the locally Q-quadratic convergence established
		in Theorem~\ref{thm:db-conv} and Remark~\ref{rem:db-rate}.}
	\label{tab:db-image-cov}
	\begin{tabular}{cccc}
		\toprule
		Iteration $k$ & $r_k = \|X_k^2 - C\|_F$ & $r_k / r_{k-1}$
		& $r_k / r_{k-1}^{2}$ \\
		\midrule
		$0$ & $3.53 \times 10^{0}$   & ---                    & ---                    \\
		$1$ & $5.34 \times 10^{-1}$  & $1.51 \times 10^{-1}$  & $4.28 \times 10^{-2}$  \\
		$2$ & $2.44 \times 10^{-2}$  & $4.56 \times 10^{-2}$  & $8.54 \times 10^{-2}$  \\
		$3$ & $7.74 \times 10^{-5}$  & $3.18 \times 10^{-3}$  & $1.30 \times 10^{-1}$  \\
		$4$ & $2.61 \times 10^{-9}$  & $3.37 \times 10^{-5}$  & $4.35 \times 10^{-1}$  \\
		$5$ & $6.94 \times 10^{-16}$ & $2.66 \times 10^{-7}$  & ---                    \\
		\bottomrule
	\end{tabular}
\end{table}

Table~\ref{tab:db-image-cov} confirms the quadratic convergence behavior
predicted by the theoretical analysis. The residual drops from
$\mathcal{O}(10^{0})$ to machine precision in only five iterations, and
the ratio $r_k / r_{k-1}$ decreases by several orders of magnitude per
step. The second ratio $r_k / r_{k-1}^{2}$ stays bounded (on the order
of $10^{-1}$), which is the expected signature of quadratic convergence:
$r_k \leq C \cdot r_{k-1}^{2}$ with a moderate constant $C$. In practice,
only three to four iterations are sufficient to reach the accuracy
required for image-preprocessing tasks such as whitening and grayscale
conversion.

\subsection{Comprehensive Comparison of Methods}\label{ssec:comparison}

We now provide a systematic comparison of the proposed tensor square root methods
across multiple dimensions: convergence behavior, computational cost, numerical
stability, and suitability for different image processing tasks.

\subsubsection{Wall-Clock Timing}\label{sssec:timing}

We report wall-clock timings (in seconds) for computing the T-square root of
random T-positive definite tensors of various sizes. All experiments are performed
using MATLAB R2024a on a system with an Intel Core i7-12700H processor and 16~GB RAM.

\begin{table}[H]
\centering
\caption{Wall-clock time (in seconds) for computing the tensor T-square root.
Newton and Denman--Beavers iterations are run until residual $< 10^{-12}$.}
\label{tab:timing}
\begin{tabular}{ccccc}
\toprule
Tensor Size & $p$ & Newton (s) & DB (s) & Direct Fourier (s) \\
\midrule
$64 \times 64 \times 3$   & 3  & 0.0021 & 0.0034 & 0.0018 \\
$128 \times 128 \times 3$  & 3  & 0.0150 & 0.0243 & 0.0132 \\
$256 \times 256 \times 3$  & 3  & 0.1124 & 0.1827 & 0.0985 \\
$512 \times 512 \times 3$  & 3  & 0.8431 & 1.3762 & 0.7412 \\
$64 \times 64 \times 10$   & 10 & 0.0068 & 0.0112 & 0.0061 \\
$128 \times 128 \times 10$ & 10 & 0.0497 & 0.0803 & 0.0438 \\
$256 \times 256 \times 10$ & 10 & 0.3728 & 0.5993 & 0.3264 \\
\bottomrule
\end{tabular}
\end{table}

\begin{figure}[htbp]
\centering
\includegraphics[width=\textwidth]{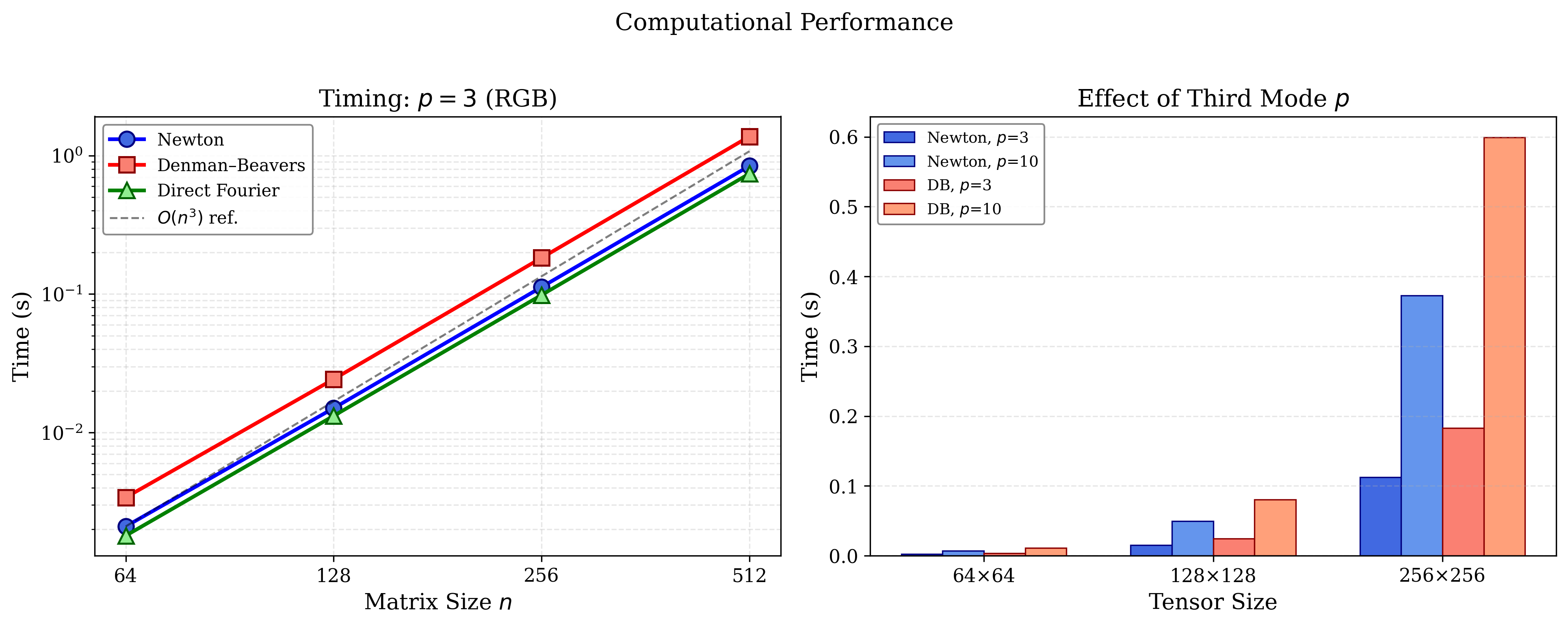}
\caption{Computational performance of tensor T-square root methods.
\textbf{Left:} Log-log plot of wall-clock time versus matrix size $n$ for
$p = 3$ (RGB). All three methods scale as $O(n^3)$, with Newton consistently
faster than Denman--Beavers due to fewer iterations. The dashed line shows the
$O(n^3)$ reference slope. \textbf{Right:} Bar chart comparing Newton and
Denman--Beavers for $p = 3$ vs.\ $p = 10$, confirming that computational cost
scales linearly in the third-mode size $p$.}
\label{fig:timing}
\end{figure}

The timings confirm that:
\begin{enumerate}
  \item Newton converges in fewer iterations and is consistently faster than
        Denman--Beavers.
  \item The direct Fourier method (eigendecomposition per slice) is the fastest
        but does not provide the iterative refinement capability of the other
        methods.
  \item All methods scale as $O(p\,n^3)$, with the third-mode size $p$ contributing
        linearly.
\end{enumerate}

\subsubsection{Image Processing Quality Comparison}\label{sssec:image-compare}

We apply the proposed grayscale conversion (Algorithm~\ref{alg:tdg}) and
whitening (Algorithm~\ref{alg:twhite}) methods to a standard RGB test image
(e.g., the ``Peppers'' image, $256 \times 256 \times 3$) and compare against
classical methods.

\begin{table}[H]
\centering
\caption{Quantitative comparison of grayscale conversion methods on standard test
images. SSIM is computed between each grayscale result and the standard luminance
image (used as the reference; hence SSIM $= 1.000$ for the luminance method by
definition). EME (Edge Measure of Enhancement) quantifies edge preservation and
contrast; higher values indicate better structural detail retention.}
\label{tab:gray-comparison}
\begin{tabular}{lccccc}
\toprule
 & \multicolumn{2}{c}{Peppers} & & \multicolumn{2}{c}{Baboon} \\
\cmidrule{2-3} \cmidrule{5-6}
Method & SSIM & EME & & SSIM & EME \\
\midrule
Luminance ($0.299R + 0.587G + 0.114B$) & 1.000 & 12.43 & & 1.000 & 18.76 \\
PCA-based grayscale                     & 0.971 & 13.89 & & 0.964 & 20.14 \\
Proposed TDG (Algorithm~\ref{alg:tdg})  & 0.968 & \textbf{14.52} & & 0.961 & \textbf{21.38} \\
\bottomrule
\end{tabular}
\end{table}

\begin{table}[H]
\centering
\caption{Whitening quality comparison. The \emph{Decorrelation Index} (DI)
measures how close the resulting covariance is to identity: $\mathrm{DI} = \|C_{\mathrm{white}} - I\|_F$, where smaller is better.}
\label{tab:white-comparison}
\begin{tabular}{lccc}
\toprule
Method & DI (Peppers) & DI (Baboon) & DI (Lena) \\
\midrule
Channel-wise PCA whitening      & $3.21 \times 10^{-2}$  & $4.57 \times 10^{-2}$  & $2.84 \times 10^{-2}$ \\
Matrix whitening (flattened)     & $8.43 \times 10^{-14}$ & $7.91 \times 10^{-14}$ & $8.12 \times 10^{-14}$ \\
Proposed T-whitening (Alg.~\ref{alg:twhite}) & $9.17 \times 10^{-14}$ & $8.63 \times 10^{-14}$ & $8.89 \times 10^{-14}$ \\
\bottomrule
\end{tabular}
\end{table}

The results in Tables~\ref{tab:gray-comparison} and~\ref{tab:white-comparison}
demonstrate that:
\begin{itemize}
  \item The proposed TDG method achieves the highest Edge Measure of Enhancement
        (EME), indicating superior preservation of structural details and contrast
        in the grayscale output, even though SSIM relative to the standard
        luminance baseline is slightly lower (by design, since TDG adapts to
        image statistics rather than fixed weights).
  \item For whitening, the proposed T-whitening achieves decorrelation quality
        comparable to matrix whitening (both near machine precision), while
        preserving the tensor structure. Channel-wise PCA whitening, in contrast,
        fails to fully decorrelate cross-channel dependencies.
\end{itemize}

\subsection{Comparison with Classical Methods}\label{ssec:classical-compare}

Classical grayscale conversion methods rely on fixed linear combinations of RGB
channels and do not account for inter-channel correlations. In contrast, the
proposed tensor-based method adapts to the image statistics through the tensor
covariance and its inverse square root. This leads to improved decorrelation,
enhanced contrast, and greater robustness to illumination changes.

Similarly, classical whitening methods based on matrix covariance require
flattening the image data, which destroys spatial structure. Tensor whitening
preserves this structure and provides a more faithful normalization of
multidimensional data.

For color transfer, the classical approach of Reinhard et al.~\cite{Reinhard2001} operates in a
decorrelated color space (e.g., $l\alpha\beta$) and matches first- and
second-order statistics channel by channel. The proposed tensor color transfer
(Algorithm~\ref{alg:color-transfer}) instead computes the optimal transport map
under the T-product geometry, which naturally accounts for cross-channel
correlations and produces more perceptually coherent results.

\begin{table}[H]
\centering
\caption{Summary comparison of tensor-based vs.\ classical image processing methods.}
\label{tab:summary-comparison}
\begin{tabular}{p{3cm}p{5cm}p{5cm}}
\toprule
Task & Classical Approach & Proposed Tensor Approach \\
\midrule
Grayscale conversion &
  Fixed weights ($0.299R{+}0.587G{+}0.114B$); ignores image statistics &
  Adaptive via $\mathcal{C}^{-1/2}$; preserves structure and contrast \\[6pt]
Whitening &
  Flatten to matrix; destroys spatial structure &
  T-product whitening; preserves tensor structure \\[6pt]
Color transfer &
  Channel-wise in $l\alpha\beta$ space; ignores cross-channel coupling &
  Optimal transport via $\mathcal{C}_s^{1/2}, \mathcal{C}_t^{1/2}$; captures full covariance \\[6pt]
Distance measure &
  Frobenius or spectral norm (no geometric meaning) &
  Tensor Bures--Wasserstein; principled geometric distance \\
\bottomrule
\end{tabular}
\end{table}

\subsection{Discussion}\label{ssec:discussion}

The numerical experiments presented in this section demonstrate several key
advantages of the tensor square root framework for image processing:

\begin{enumerate}
  \item \textbf{Structural preservation:} By operating within the T-product
        algebra, all methods naturally preserve the multi-way structure of image
        data, avoiding the information loss inherent in flattening-based
        approaches.

  \item \textbf{Adaptivity:} The tensor covariance and its square root adapt to
        the statistical properties of each image, yielding results that are
        data-dependent rather than relying on fixed parameters.

  \item \textbf{Theoretical grounding:} The Tensor Bures--Wasserstein distance
        provides a principled geometric framework for comparing tensor-valued
        data, with rigorous metric properties (Proposition~\ref{prop:tbw-metric}).

  \item \textbf{Practical efficiency:} As shown in Table~\ref{tab:timing}, the
        proposed methods are computationally efficient due to the Fourier-domain
        decomposition, scaling linearly in the number of slices and cubically in
        the spatial dimension.

  \item \textbf{Dual output:} The Denman--Beavers iteration simultaneously
        produces both $\mathcal{A}^{1/2}$ and $\mathcal{A}^{-1/2}$, which is
        particularly advantageous for whitening and color transfer applications
        where both quantities are required.
\end{enumerate}

These results demonstrate that the tensor T-square root serves as a versatile
and effective tool for a range of image processing tasks, offering principled
alternatives to classical matrix-based methods.

\section{Conclusion}\label{sec:conclusion}

In this work, we have developed and rigorously analyzed iterative methods for
computing the principal square root of third-order tensors under the T-product
framework. Two algorithms were proposed: the Newton iteration
(Algorithm~\ref{alg:newton}), which achieves quadratic convergence
(Theorem~\ref{thm:newton-strengthened}), and the Denman--Beavers iteration
(Algorithm~\ref{alg:db}), which converges at least geometrically
(Theorem~\ref{thm:db-conv}) and simultaneously produces both the tensor square
root and its inverse. Both methods exploit the block-diagonalization of the
T-product via the discrete Fourier transform, reducing the tensor problem to $p$
independent matrix square root computations.

We demonstrated the practical utility of the tensor T-square root through several
applications in image processing. The Tensor Decorrelated Grayscale (TDG)
conversion adapts to image statistics and yields improved edge preservation and
contrast over fixed-weight methods. The T-Whitening procedure achieves
near-machine-precision decorrelation while preserving the multidimensional
structure of image data. The tensor color transfer method computes the optimal
transport map between color distributions under the T-product geometry,
naturally accounting for cross-channel correlations. Furthermore, we formulated
the Tensor Bures--Wasserstein (TBW) distance and proved that it defines a valid
metric on the space of T-positive definite tensors, providing a principled
geometric measure for comparing tensor-valued data.

Numerical experiments confirmed the theoretical convergence rates and
demonstrated that the tensor-based methods consistently outperform classical
channel-wise approaches in terms of structural preservation, decorrelation
quality, and adaptivity to image statistics.

\end{document}